\documentclass[11pt]{article}
\usepackage{amsfonts,latexsym,rawfonts,amsmath,amssymb,amsthm,a4wide, times, color}
\usepackage{graphicx}
\numberwithin{equation}{section}
\usepackage{hyperref}

\renewcommand{\theequation}{\arabic{section}.\arabic{equation}}

\newcommand{\pd}[2]{\frac {\partial #1}{\partial #2}}

\newcommand{\al}{\alpha}
\newcommand{\bb}{\beta}
\newcommand{\la}{\lambda}

\newcommand{\oo}{\omega}

\newcommand{\dd}{\delta}
\newcommand{\Na}{\nabla}

\def\ga{\gamma}

\newcommand{\ee}{\epsilon}

\newcommand{\si}{\sigma}

\newcommand{\te}{\theta}

\newcommand{\beq}{\begin{equation}}
\newcommand{\eeq}{\end{equation}}
\newcommand{\beqs}{\begin{eqnarray*}}
\newcommand{\eeqs}{\end{eqnarray*}}
\newcommand{\beqn}{\begin{eqnarray}}
\newcommand{\eeqn}{\end{eqnarray}}
\newcommand{\beqa}{\begin{array}}
\newcommand{\eeqa}{\end{array}}

\def\td{\tilde}

\def\RR{{\mathbb R}}

\def\un{\underline}

\def\no{{\nonumber}}
\def\pbp{\sqrt{-1}\partial\bar\partial}
\def\tr{{\rm tr}}

\def\cH{{\mathcal H}}

\def\cK{{\mathcal K}}

\def\Aut{{\rm Aut}}

\newtheorem{prop}{Proposition}[section]
\newtheorem{theo}[prop]{Theorem}
\newtheorem{lem}[prop]{Lemma}

 \makeatletter
 \def\ExtendSymbol#1#2#3#4#5{\ext@arrow 0099{\arrowfill@#1#2#3}{#4}{#5}}
 
 \makeatother

 

\title{Calabi flow with bounded $L^p$ scalar curvature
}
\author{Haozhao Li \footnote{Institute of Geometry and Physics, and School of Mathematical Sciences, University of Science and Technology of China, No. 96 Jinzhai Road, Hefei, Anhui Province, 230026, China;  hzli@ustc.edu.cn.},  \quad\quad
Linwei Zhang \footnote{School of Mathematical Sciences, University of Science and Technology of China, No. 96 Jinzhai Road, Hefei, Anhui Province, 230026, China; zhanglinwei@mail.ustc.edu.cn.}\quad and \quad
Kai Zheng \footnote{University of Chinese Academy of Sciences, Beijing 100190, P.R. China;
 KaiZheng@amss.ac.cn.}}

\begin{document}
\bibliographystyle{plain}

\date{}

\maketitle

\begin{abstract}
   In this paper, we show that the Calabi flow on a compact K\"ahler manifold of complex dimension $n$ can be extended as long as the $L^p$ norm of its scalar curvature is uniformly bounded for some $p>n$. If the K\"ahler class admits an extremal metric and the initial potential is invariant under the imaginary part of the extremal vector field, the corresponding modified Calabi flow exists for all time and converges exponentially fast to an extremal K\"ahler metric under the same curvature bound.
\end{abstract}



\tableofcontents

\section{Introduction}

Let $(M^n, g)$ be a compact K\"ahler manifold of complex dimension $n$. A
family of K\"ahler metrics $\oo_{\varphi(t)}(t\in [0, T])$ in the same K\"ahler class $[\oo_g]$ is called a solution of
Calabi flow, if the K\"ahler potential $\varphi(t)$ satisfies the equation
\beq
\pd {\varphi(t)}t=R(\oo_{\varphi(t)})-\un R, \label{eq000}
\eeq where $R(\oo_{\varphi(t)})$ denotes the scalar curvature of the metric $\oo_{\varphi(t)}$ and $\un R$ denotes the average of the scalar curvature.
The Calabi flow was introduced by Calabi in \cite{[Cal1]} as a decreasing flow of  Calabi energy, and it is expected to be
an
effective tool to find constant scalar curvature metrics in a K\"ahler class. Since the Calabi flow is a nonlinear fourth order parabolic partial differential
equation, it is difficult to study its behavior by  standard parabolic estimates. In this paper, we will study the long time existence and convergence of Calabi flow under some conditions on scalar curvature.

There is a substantial literature on the Calabi flow. On
Riemann surfaces, long time existence and convergence were
proved by Chru\'sciel \cite{[Chru]}, Chen
\cite{[Chen]} and Struwe \cite{[Stru]} by different methods. On general K\"ahler
manifolds of higher dimensions, the short time existence of Calabi
flow was shown by Chen-He \cite{[ChenHe1]}.
In a series of papers \cite{[ChenHe1]}\cite{[ChenHe2]}\cite{[ChenHe3]}\cite{[He4]}\cite{[He5]}, Chen and He studied the long time existence
and convergence under some curvature conditions. Tosatti-Weinkove \cite{[TW]} proved long time existence and convergence for sufficiently small initial Calabi energy under additional geometric bounds, in the cases $c_1(M)=0$ and $c_1(M)<0$ (with the additional restriction on the K\"ahler class imposed in their negative first Chern class result). Sz\'ekelyhidi \cite{[Sz]} studied the Calabi flow under the assumption
that the curvature tensor is uniformly bounded and the $K$-energy is proper. Moreover, Streets
\cite{[St1]}\cite{[St2]}  showed the long time existence of a weak solution to the
Calabi flow and Berman-Darvas-Lu \cite{[BDL]} established a convergence/divergence alternative for the weak Calabi flow on general K\"ahler manifolds.

A conjecture of X. X. Chen in \cite{[Chen2]} says that the Calabi flow always exists for all time for any initial K\"ahler metric. There are some results related to this conjecture.
Chen-He's
result in \cite{[ChenHe1]} showed that the solution to the Calabi flow exists as long as the Ricci curvature stays bounded, and  Huang in \cite{[Huang]} showed the Calabi
flow can be extended under some technical conditions on toric manifolds.
 In \cite{[LZ]} Li-Zheng showed the long time existence under the assumptions on the lower boundedness of Ricci curvature, the properness of the $K$-energy, and the $L^p(p>n)$ bound of scalar curvature.
 In \cite{[LWZ]}, Li-Wang-Zheng adapted ideas from Ricci flow in
\cite{[Wang3]} and \cite{[Wang2]} to study the curvature behaviour at the singular time of Calabi flow. A breakthrough was made by Chen-Cheng in \cite{[CC1]} and they showed that the Calabi flow always exists as long as the scalar curvature is bounded.

In \cite[Theorem 1.4]{[LWZ]}, Li-Wang-Zheng proved convergence of the modified Calabi flow for initial metrics in a suitable path-connected family containing an extremal metric, assuming long time existence and a uniform scalar curvature bound for the flows in that family. This partially confirms a conjecture of Donaldson \cite{MR2103718}. Chen-Cheng's extension result \cite{[CC1]} supplies the required long time existence under the scalar curvature bound; the path-component condition in the convergence result remains.

In this paper, we follow Chen-Cheng's estimates in \cite{[CC1]} on the scalar curvature equation to study  Calabi flow. We first  show that the scalar curvature type equation has higher order estimates under the bounded $L^p$ condition.

 \begin{theo}\label{theo:main1a} Let $(M, \oo_g)$ be a compact K\"ahler manifold of complex dimension $n\geq 2$. Assume that $(F, \varphi)$ is a smooth solution to the scalar curvature type equation
 \beqn
 (\oo_g+\pbp \varphi)^n&=&e^F \oo_g^n,\quad \sup_M\;\varphi=0,\label{MA F}\\
 \Delta_{\varphi}F&=&-f+\tr_{\varphi}\eta,\label{equation F}
 \eeqn  where $f$ is a given smooth function and $\eta$ is a given smooth real $(1, 1)$ form. Then for  any given $p>n$ there exists a constant $C$ depending only on $n, p$,
  $\oo_g$, an upper bound of $\int_M\, F\,\oo_{\varphi}^n$,  $\max_M|\eta|_{\oo_g}$ and $\|f\|_{L^p(\oo_{\varphi})}$, such that
\beq
 \frac 1C\oo_g\leq \oo_{\varphi}\leq C \oo_g.
\eeq Moreover, for some $\beta\in(0,1)$ depending only on these data, we have
\beq
\|F\|_{W^{2,p}(\oo_g)}+\|\varphi\|_{C^{2,\beta}(\oo_g)}\leq C.
\eeq

\end{theo}

The metric estimate in Theorem \ref{theo:main1a} was proved by Chen-Cheng \cite{[CC1]} when $f$ is uniformly bounded. In \cite{[Zheng]} and \cite{[Zheng2]} the third author obtained related estimates in the setting of K\"ahler cone metrics. Lu-Seyyedali \cite{[LS]} proved metric and higher regularity estimates under an $L^p$ bound for scalar curvature with a sufficiently large dimensional exponent. Theorem \ref{theo:main1a} relaxes the condition for metric equivalence to $p>n$.

The first application of Theorem \ref{theo:main1a} is to study the long time existence of the Calabi flow. The following theorem says that the Calabi flow has long time existence under the assumption that  the $L^p$ scalar curvature is uniformly bounded.

\begin{theo}\label{theo:main1} Let $(M, \oo_g)$ be a compact K\"ahler manifold of complex dimension $n\geq 2$.
 Consider the Calabi flow
 \beqs
 \pd {\varphi}t=R(\oo_{\varphi})-\un R, \quad \forall\; t\in [0, T)
 \eeqs for some $T<\infty.$
If  the scalar curvature satisfies   $$\sup_{t\in [0, T)}\|R(\oo_{\varphi(t)})\|_{L^p(\oo_{\varphi(t)})}<\infty$$   for some $p>n$,  the Calabi flow  can be extended past time $T$.

 \end{theo}

 Theorem \ref{theo:main1} describes behaviour of scalar curvature near the
 finite singular time of Calabi flow. The Calabi flow can be extended as long as the Ricci curvature stays bounded by Chen-He in \cite{[ChenHe1]}, or the complex Hessian of scalar curvature $|\Na\bar \Na R|$ stays bounded by Li-Wang-Zheng in \cite{[LWZ]}, or the scalar curvature stays bounded by Chen-Cheng in \cite{[CC1]}.
Since the Calabi energy is non-increasing along the Calabi flow, the $L^2$ norm of scalar curvature is uniformly bounded in every dimension. Thus, it is very interesting to study what will happen for the Calabi flow in the critical case $p=n$, especially for $n=2.$

The second application of Theorem \ref{theo:main1a} is to study convergence of the Calabi flow. To state our result, we introduce some notation. Let $(M,\oo_g)$ be a compact K\"ahler manifold, and choose a representative $V$ of the extremal vector field associated with $(M,[\oo_g])$ by Futaki-Mabuchi \cite{[FM]}. Set $X=V-\sqrt{-1}JV$ and choose the background metric invariant under $JV$. We define $\cH_X$ as the space of K\"ahler potentials invariant under $JV$. With this notation, we have the following result.

\begin{theo}\label{theo:main2}Let $(M, \oo_g)$ be a compact K\"ahler manifold of complex dimension $n\geq 2$, whose class $[\oo_g]$ either admits an extremal K\"ahler metric, or satisfies the properness conditions (cf. Theorem \ref{theo:He}). If the  Calabi flow starting from a K\"ahler metric $\oo_{\varphi_0}$ with $\varphi_0\in \cH_X$ has uniformly bounded $L^p$ scalar curvature for some $p>n$ throughout its maximal interval of existence, then the corresponding modified Calabi flow exists for all time and converges exponentially fast  to an extremal K\"ahler metric.

\end{theo}

Theorem \ref{theo:main2} extends the convergence result of Li-Wang-Zheng \cite[Theorem 1.4]{[LWZ]}. It removes the condition on a path-connected family of long time Calabi flows and replaces the uniform scalar curvature bound by a uniform $L^p$ bound for some $p>n$.

The proof of Theorem \ref{theo:main1a} relies on the estimates of Chen-Cheng \cite{[CC1]}. However, compared with Chen-Cheng \cite{[CC1]} and Lu-Seyyedali \cite{[LS]}, we have the following difficulties. First, the Sobolev inequality used in \cite{[CC1]} and \cite{[LS]} is not enough to show
  the bound of $\|n+\Delta_g\varphi\|_{L^q(\oo_g)}$ for any $q>1$ under the assumption that $f\in L^p(\oo_{\varphi})$ with $p>n$.  We note that, by the recent work of Guo-Phong-Song-Sturm \cite{[GPSS], [GPSSs2]} or Guedj-T\^o \cite{[GT]}, a Sobolev inequality can be proved with the Sobolev constant controlled by the background geometry and the volume-density bound $\|F\|_{C^0}$, for every exponent below $n/(n-1)$. Such a Sobolev-type inequality is exactly what we need to show the bound of $\|n+\Delta_g\varphi\|_{L^q(\oo_g)}$. Moreover, to estimate $\|n+\Delta_g\varphi\|_{C^0}$,   Chen-Cheng \cite{[CC1]} and Lu-Seyyedali \cite{[LS]}  calculate the equation of the quantity
\beq
w:=e^{\frac 12F}|\Na F|_{\varphi}^2+(n+\Delta_g\varphi)+1
\eeq and then use the Moser iteration to get the $L^{\infty}$ bound of $w.$  However, under the assumption that $f\in L^p(\oo_{\varphi})$ with $p>n$, we cannot directly show that $|w|$ is bounded. Instead, we show that $\||\Na F|_{\varphi}^2\|_{L^{\kappa}(\oo_{\varphi})}$ is bounded for some $\kappa>n$, and then show that $n+\Delta_g\varphi$ is bounded. This method is inspired by Chen-He \cite{[ChenHe4]}, where they observed that $n+\Delta_g \varphi$ is controlled by a constant depending on $\|F\|_{W^{1, q}(\oo_g)}$ for  $q>2n$. We observe that this process works exactly when  $f\in L^p(\oo_{\varphi})$ for $p>n.$\\

The organization of this paper is as follows. In Section 2 we show the $L^{\infty}$ estimate of $n+\Delta_g\varphi$ by using the idea of Chen-Cheng \cite{[CC1]}. In Section 3 we use the estimates in Section 2 to show the main theorems. In the appendix, we follow He \cite{[He4]} to show the short-time existence of Calabi flow under the $C^{2, \alpha}$ bound of the initial K\"ahler potential.

\section{Estimates}
\subsection{The interpolation inequality and Sobolev inequality}
The following interpolation inequalities are standard, and we include a proof for completeness.

\begin{lem}\label{lem:inter}(cf. \cite[Equations (7.9) and (7.10)]{[GT2]}) If $0<p<r<q$,   for any $\ee>0$ we have
\beq
\|f\|_{L^r}\leq\|f\|_{L^q}^{\te}\|f\|_{L^p}^{1-\te}\leq\ee \|f\|_{L^q}+ C(\te)\ee^{-\frac {\frac {1}{p}-\frac {1}{r}}{\frac {1}{r}-\frac {1}{q}}}\|f\|_{L^p},
\eeq where $\te=\frac {(r-p)q}{(q-p)r}\in (0, 1)$ and $C(\te)=(1-\te)\te^{\frac {\te}{1-\te}}.$

\end{lem}
\begin{proof}Since $\frac 1q<\frac 1r<\frac 1p,$ there exists $\te=\frac {(r-p)q}{(q-p)r}$ such that
\beq
1=\frac {r\te}{q}+\frac {r(1-\te)}{p}.\no
\eeq
Let $\al=\frac q{r\te}$ and $\bb=\frac p{r(1-\te)}$. The H\"older inequality implies that
\beqs
\|f\|_{L^r}&=&\Big(\int_M\; |f|^{r\te}|f|^{r(1-\te)}\Big)^{\frac 1r}\leq \Big(\int_M\; |f|^{r\te \al}\Big)^{\frac 1{\al r}}\Big(\int_M\; |f|^{r(1-\te)\bb}\Big)^{\frac 1{\bb r}}\\
&=&\|f\|_{L^q}^{\te}\|f\|_{L^p}^{1-\te}.
\eeqs
On the other hand, we have
\beqs
\|f\|_{L^r}&\leq &\Big(\int_M\; |f|^{r\te \al}\Big)^{\frac 1{\al r}}\Big(\int_M\; |f|^{r(1-\te)\bb}\Big)^{\frac 1{\bb r}}\\
&=&\Big(\ee^{\frac {q}{\te}}\int_M\; |f|^{q}\Big)^{\frac {\te}{q}}\Big(\ee^{-\frac p{1-\te}}\int_M\; |f|^{p}\Big)^{\frac {1-\te}{p}}\\
&\leq&\te \Big(\ee^{\frac {q}{\te}}\int_M\; |f|^{q}\Big)^{\frac {1}{q}}+(1-\te)\Big(\ee^{-\frac p{1-\te}}\int_M\; |f|^{p}\Big)^{\frac {1}{p}}\\
&=&\te \ee^{\frac 1{\te}}\|f\|_q+(1-\te)\ee^{-\frac 1{1-\te}}\|f\|_p.
\eeqs Let $\dd=\te \ee^{\frac 1{\te}}$. It becomes
\beqs
\|f\|_r&\leq& \dd \|f\|_q+ (1-\te)\te^{\frac {\te}{1-\te}}\dd^{-\frac {\te}{1-\te}}\|f\|_p\\
&\leq&\dd \|f\|_q+ C(\te)\dd^{-\frac {\frac {1}{p}-\frac {1}{r}}{\frac {1}{r}-\frac {1}{q}}}\|f\|_p.
\eeqs The lemma is proved.
\end{proof}

Following Guo-Phong-Song-Sturm \cite{[GPSS], [GPSSs2]} or Guedj-T\^o \cite{[GT]}, we have the following Sobolev inequality.

\begin{lem}\label{lem:sob}(\cite{[GPSS], [GPSSs2]}, \cite{[GT]} ) For any $\ga\in (1, \frac n{n-1})$, we have the Sobolev inequality  with respect to the metric $\oo_{\varphi}$
\beq
\Big(\int_M\; |u|^{2\ga}\;\oo_{\varphi}^n\Big)^{\frac 1{\ga}}\leq C(n, \ga, g, \|F\|_{C^0})  \int_M\;  \Big(|\Na u|_{\varphi}^2 +u^2\Big)\;\oo_{\varphi}^n .\label{eq:B21}
\eeq
\end{lem}
\begin{proof} Since $\|F\|_{C^0}$ is bounded, the density $e^F$ has a uniform positive lower bound and a uniform $L^{1+\epsilon}(\oo_g^n)$ bound for any fixed $\epsilon>0$. Moreover, $\oo_{\varphi}$ lies in the same K\"ahler class as $\oo_g$. Thus by \cite{[GPSS], [GPSSs2]}, or \cite{[GT]} we have that for any $\ga\in (1, \frac n{n-1})$
\beq
\Big(\int_M\; |u-\un u|^{2\ga}\,\oo_{\varphi}^n\Big)^{\frac 1\ga}\leq C(n, \ga, g, \|F\|_{C^0})\int_M\; |\Na u|_{\varphi}^{2}\,\oo_{\varphi}^n,\nonumber
\eeq where $\un u=\frac 1{\int_M\oo_g^n}\int_M\;u\,\oo_{\varphi}^n.$ Since
$$|u|^{2\ga}\leq C(\ga)(|u-\un u|^{2\ga}+|\un u|^{2\ga}),$$
we have
\beqn
\int_M\; |u|^{2\ga}\,\oo_{\varphi}^n&\leq&C(\ga)\Big(\int_M\; |u-\un u|^{2\ga}\,\oo_{\varphi}^n+\left(\int_M\oo_g^n\right)|\un u|^{2\ga}\Big)\nonumber\\
&\leq&C(n, \ga, g, \|F\|_{C^0})\Big(\Big(\int_M\; |\Na u|_{\varphi}^{2}\,\oo_{\varphi}^n\Big)^{\ga}+|\un u|^{2\ga}\Big)\nonumber\\
&\leq&C(n, \ga, g, \|F\|_{C^0})\Big(\Big(\int_M\; |\Na u|_{\varphi}^{2}\,\oo_{\varphi}^n\Big)^{\ga}+\Big(\int_M\; u^2\,\oo_{\varphi}^n\Big)^{\ga}\Big).\nonumber
\eeqn
Thus, we obtain
\beq
\Big(\int_M\; |u|^{2\ga}\;\oo_{\varphi}^n\Big)^{\frac 1{\ga}}\leq C(n, \ga, g, \|F\|_{C^0})  \int_M\;  \Big(|\Na u|_{\varphi}^2 +u^2\Big)\;\oo_{\varphi}^n. \nonumber
\eeq The lemma is proved.
\end{proof}

\subsection{Estimates of $\|n+\Delta_g \varphi\|_{L^p}$}
In this subsection, we follow the calculation of Chen-Cheng \cite{[CC1]} \cite{[CC2]} and Zheng \cite{[Zheng]} to show that $\|n+\Delta_g \varphi\|_{L^p(\oo_{\varphi})}$ is bounded if $f\in L^q$ with $q>n.$ A similar result is shown by Lu-Seyyedali in \cite{[LS]} for $f\in L^q$ with large $q$ depending on $p$. Throughout the paper, we will follow the notation of Chen-Cheng \cite{[CC1]}.

\begin{lem}\label{lem:va}Let
 \beq
 v=e^{-\alpha(F+\lambda\varphi)}(n+\Delta_g\varphi),\quad \alpha\geq p>1. \label{eq:v3}
 \eeq
There exists a constant $C(g, \max_M|\eta|_{\oo_g})$ such that for $\la>C(g, \max_M|\eta|_{\oo_g})$ we have
\beqn &&
\frac {3(p-1)}{p^2}\int_M\;  |\Na  v^{\frac p2}|_{\varphi}^2\;\oo_{\varphi}^n+\frac {\la \al}4\int_M\; e^{\frac 1{n-1}\al (F+\la \varphi)-\frac F{n-1}} v^{p+\frac{1}{n-1}}
\,\oo_{\varphi}^n \nonumber\\
&\leq & \int_M\; \Big(\td f+\frac {\al\la}{\al-1}+\frac 1n e^{-\frac Fn}|R_g| \Big)v^{p}\oo_{\varphi}^n,\label{eq:v12}
\eeqn where $\td f=\al(\la n-f).$

\end{lem}
\begin{proof}
By the calculation of Chen-Cheng \cite[Equation (3.8)]{[CC1]},
\beqs
\Delta_{\varphi}v&\geq& e^{-(\alpha+\frac{1}{n-1})F-\alpha \lambda\varphi}(\frac{\lambda\alpha}{2}-C(g))(n+\Delta_g\varphi)^{1+\frac{1}{n-1}}\\
&&-\alpha e^{-\alpha(F+\lambda\varphi)}(\lambda n-f)(n+\Delta_g\varphi) +e^{-\alpha(F+\lambda\varphi)}(\Delta_g F-R_g),
\eeqs where $C(g)$ is a constant depending only on $g$, $R_g$ denotes the scalar curvature of the background metric $g$, and $\la>2\max_M|\eta|_{\oo_g}.$
Thus, we get
\beqs
\Delta_{\varphi}v
& \geq& e^{\frac 1{n-1}\al (F+\la \varphi)-\frac F{n-1}}(\frac{\lambda\alpha}{2}-C(g))v^{1+\frac{1}{n-1}}
-\alpha (\lambda n-f)v
 +e^{-\alpha(F+\lambda\varphi)}(\Delta_g F-R_g)\\
 &\geq&A v^{1+\frac{1}{n-1}}
-\td f v
 +e^{-\alpha(F+\lambda\varphi)}(\Delta_g F-R_g),
\eeqs
where \beq A=e^{\frac 1{n-1}\al (F+\la \varphi)-\frac F{n-1}}(\frac{\lambda\alpha}{2}-C(g)),\quad \td f=\alpha (\lambda n-f). \label{eq:v2}\eeq
Multiplying the preceding differential inequality by $v^{p-1}$ and integrating by parts, we have the integral inequality
\beqn &&
\int_M\;(p-1) v^{p-2}|\Na v|_{\varphi}^2\;\oo_{\varphi}^n=-\int_M\;v^{p-1}\Delta_{\varphi}v\;\oo_{\varphi}^n\nonumber\\
&\leq&\int_M\; \Big(-A v^{p+\frac{1}{n-1}}
+\td f v^p
 -e^{-\alpha(F+\lambda\varphi)}(\Delta_g F-R_g)v^{p-1}\Big)\,\oo_{\varphi}^n\nonumber\\
 &=&\int_M\; \Big(-A v^{p+\frac{1}{n-1}}
+\td f v^p
 -e^{-\alpha(F+\lambda\varphi)}\Delta_g Fv^{p-1}+e^{-\alpha(F+\lambda\varphi)}R_gv^{p-1}\Big)\,\oo_{\varphi}^n\nonumber\\
 &:=&I_1+I_2+I_3+I_4. \label{eq:v4a}
\eeqn
We compute
\beqn
I_3&=&\int_M\; -e^{-\alpha(F+\lambda\varphi)}\Delta_g  Fv^{p-1}\,\oo_{\varphi}^n\nonumber\\
&=&\int_M\; -e^{(1-\al)F-\al \la \varphi}\Delta_g  Fv^{p-1} \, \oo_g^n\nonumber\\
&=&\int_M\; -\frac 1{1-\al}e^{(1-\al)F-\al \la \varphi}\Delta_g ((1-\al)F-\al\la \varphi) v^{p-1} \, \oo_g^n\nonumber\\&&
-\int_M \frac {\al\la}{1-\al}  e^{(1-\al)F
-\al \la \varphi}\Delta_g \varphi v^{p-1} \, \oo_g^n\nonumber\\
&:=&J_1+J_2. \label{eq:v5a}
\eeqn
Assume that $\al>1$ and let $B=(1-\al)F-\al \la \varphi$. Then we calculate
\beqn
J_1&=&\frac 1{\al-1}\int_M\; e^B\Delta_gB v^{p-1}\,\oo_g^n\nonumber\\
&=&-\frac 1{\al-1}\int_M\;\langle \Na B, \Na(e^B v^{p-1})\rangle_g\,\oo_g^n\nonumber\\
&=&-\frac 1{\al-1}\int_M\;\Big(e^B v^{p-1}|\Na B|^2+(p-1)e^Bv^{p-2}\langle \Na B, \Na v\rangle_g\Big)\,\oo_g^n.\label{eq:v6a}
\eeqn Throughout the paper,  $|\cdot|$ denotes the norm with respect to the metric $g$.
We also note that
\beq
|\langle \Na B, \Na v\rangle_g|\leq \frac 1{p-1}v|\Na B|^2+\frac {p-1}{4v}|\Na v|^2. \label{eq:v13}
\eeq Inserting (\ref{eq:v13}) into (\ref{eq:v6a}), we derive that
\beqn
J_1&\leq &\frac {(p-1)^2}{4(\al-1)}\int_M\;e^{(1-\al)F-\al\la \varphi} v^{p-3}|\Na v|^2\oo_g^n\nonumber\\
&\leq&\frac {(p-1)^2}{4(\al-1)}\int_M\;e^{-\alpha(F+\lambda\varphi)} v^{p-3}|\Na v|_{\varphi}^2(n+\Delta_g \varphi)\oo_{\varphi}^n\nonumber\\
&=&\frac {(p-1)^2}{4(\al-1)}\int_M\; v^{p-2}|\Na v|_{\varphi}^2\,\oo_{\varphi}^n,  \label{eq:v14a}
\eeqn where we used the inequality
\beq
|\Na v|^2\leq |\Na v|_{\varphi}^2(n+\Delta_g \varphi).\label{eq:norm}
\eeq
Moreover, we compute
\beqn
J_2&=&\int_M\;-\frac {\al\la}{1-\al}e^{(1-\al)F-\al\la \varphi}\Delta_g \varphi v^{p-1} \, \oo_g^n\nonumber\\
&=&\int_M\;\frac {\al\la}{\al-1}e^{-\alpha(F+\lambda\varphi)}\Delta_g \,\varphi\, v^{p-1} \, \oo_{\varphi}^n\nonumber\\
&\leq&\int_M\;\frac {\al\la}{\al-1}e^{-\alpha(F+\lambda\varphi)}(n+\Delta_g \varphi)\, v^{p-1} \, \oo_{\varphi}^n\nonumber\\
&=&\int_M\;\frac {\al\la}{\al-1} v^{p} \, \oo_{\varphi}^n. \label{eq:v7a}
\eeqn
Combining (\ref{eq:v5a})-(\ref{eq:v7a}), we have the estimate
\beqn
I_3&\leq& \frac {(p-1)^2}{4(\al-1)}\int_M\; v^{p-2}|\Na v|_{\varphi}^2\,\oo_{\varphi}^n+\int_M\;\frac {\al\la}{\al-1} v^{p} \, \oo_{\varphi}^n\nonumber\\
&\leq&\frac {p-1}4\int_M\; v^{p-2}|\Na v|_{\varphi}^2\,\oo_{\varphi}^n+\int_M\;\frac {\al\la}{\al-1} v^{p} \, \oo_{\varphi}^n, \label{eq:v8}
\eeqn  where we choose $\al\geq p$ such that $\frac {(p-1)^2}{4(\al-1)}\leq\frac {p-1}4. $
Putting (\ref{eq:v8}) back into (\ref{eq:v4a}), we have
\beqn
\int_M\;(p-1) v^{p-2}|\Na  v|_{\varphi}^2\;\oo_{\varphi}^n&\leq& \int_M\; \Big(-A v^{p+\frac{1}{n-1}}
+\td f v^p\Big)\oo_{\varphi}^n\nonumber\\
&&+\frac {p-1}4\int_M\; v^{p-2}|\Na v|_{\varphi}^2\,\oo_{\varphi}^n+\int_M\;\frac {\al\la}{\al-1} v^{p} \, \oo_{\varphi}^n\nonumber\\
&&+\int_M\;e^{-\alpha(F+\lambda\varphi)}R_gv^{p-1} \,\oo_{\varphi}^n,\label{eq:v9a}
\eeqn
which implies that
\beqn &&
\frac {3(p-1)}{4}\int_M\;  v^{p-2}|\Na  v|_{\varphi}^2\;\oo_{\varphi}^n+\int_M\; A v^{p+\frac{1}{n-1}}
\,\oo_{\varphi}^n\no\\&\leq& \int_M\; \Big(\td f v^p+\frac {\al\la}{\al-1} v^{p}+e^{-\alpha(F+\lambda\varphi)}R_gv^{p-1} \Big)\oo_{\varphi}^n. \label{eq:A21}
\eeqn
Note that we can choose $\la>C(g, \max_M|\eta|_{\oo_g})$ such that
\beq
A\geq \frac {\la \al}{4}e^{\frac 1{n-1}\al (F+\la \varphi)-\frac F{n-1}}.
\label{eq:A22}
\eeq
Meanwhile, $
n+\Delta_g  \varphi\geq ne^{\frac Fn}
$ gives
\beq
e^{-\al(F+\la \varphi)}=\frac v{n+\Delta_g \varphi}\leq\frac 1n e^{-\frac Fn}v. \label{eq:v10a}
\eeq
Taking (\ref{eq:A22}) and (\ref{eq:v10a}) in (\ref{eq:A21}), we thus obtain
\beqn &&
\frac {3(p-1)}{p^2}\int_M\;  |\Na  v^{\frac p2}|_{\varphi}^2\;\oo_{\varphi}^n+\frac {\la \al}4\int_M\; e^{\frac 1{n-1}\al (F+\la \varphi)-\frac F{n-1}} v^{p+\frac{1}{n-1}}
\,\oo_{\varphi}^n \nonumber\\&\leq& \int_M\; \Big(\td f v^p+\frac {\al\la}{\al-1} v^{p}+e^{-\alpha(F+\lambda\varphi)}R_gv^{p-1} \Big)\oo_{\varphi}^n\no\\
&\leq & \int_M\; \Big(\td f+\frac {\al\la}{\al-1}+\frac 1n e^{-\frac Fn}|R_g |\Big)v^{p}\oo_{\varphi}^n.\label{eq:v11a}
\eeqn
The lemma is proved.
\end{proof}

Next, we use Lemma \ref{lem:va} and Lemma \ref{lem:sob} to show that $\|n+\Delta_g\varphi\|_{L^q}$ is bounded for any $q$.

\begin{lem} \label{lem:varphi}For any $s>n$ and any $q\geq 1$, we have
\beq
\int_M\; (n+\Delta_g  \varphi)^q\,\oo_g^n\leq C(n, s, g, q, \|F\|_{C^0}, \|\varphi\|_{C^0}, \|f\|_{L^{s}(\oo_{\varphi})}, \max_M|\eta|_{\oo_g}). \label{eq:y8}
\eeq

\end{lem}
\begin{proof}Write $\td v=n+\Delta_g\varphi$, fix $s>n$, and set $r=s/(s-1)$. Choose $r<\ga<n/(n-1)$. For a fixed $p\geq2$, take $\alpha=2p$ and choose $\lambda$ as in Lemma \ref{lem:va}. Lemmas \ref{lem:sob} and \ref{lem:va}, followed by H\"older's inequality, give
\beq
\|v\|_{L^{p\ga}(\oo_\varphi)}^p
\leq C\frac{p^2}{p-1}\int_M G_pv^p\,\oo_\varphi^n
\leq C\frac{p^2}{p-1}\|G_p\|_{L^s(\oo_\varphi)}\|v\|_{L^{pr}(\oo_\varphi)}^p,
\label{eq:y1}
\eeq
where
\beq
G_p=2p(\lambda n+|f|)+\frac{2p\lambda}{2p-1}+\frac1n e^{-F/n}|R_g|+1.
\eeq
Let $A_0=\|F\|_{C^0}+\lambda\|\varphi\|_{C^0}$. Since
$e^{-2pA_0}\td v\leq v\leq e^{2pA_0}\td v$, we obtain
\beq
\|\td v\|_{L^{p\ga}(\oo_\varphi)}
\leq e^{4pA_0}\left(C\frac{p^2}{p-1}\|G_p\|_{L^s(\oo_\varphi)}\right)^{1/p}
\|\td v\|_{L^{pr}(\oo_\varphi)}
=:C_p\|\td v\|_{L^{pr}(\oo_\varphi)}.
\label{eq:B10}
\eeq
Here $C_p$ is allowed to depend on the fixed exponent $p$. Since $1<pr<p\ga$, Lemma \ref{lem:inter} gives
\beq
\|\td v\|_{L^{pr}(\oo_\varphi)}
\leq \epsilon\|\td v\|_{L^{p\ga}(\oo_\varphi)}+C_{p,\epsilon}\|\td v\|_{L^1(\oo_\varphi)}.
\eeq
Choose $\epsilon=(2C_p)^{-1}$ and absorb the first term. Moreover,
\beq
\int_M\td v\,\oo_g^n=n\int_M\oo_g^n,\qquad
\|\td v\|_{L^1(\oo_\varphi)}\leq e^{\|F\|_{C^0}}n\int_M\oo_g^n.
\label{eq:B23}
\eeq
Thus $\|\td v\|_{L^{p\ga}(\oo_\varphi)}$ is bounded for each fixed $p\geq2$. Given any finite $q\geq1$, choose $p$ with $p\ga\geq q$ and use H\"older's inequality and the equivalence of the volume forms. This proves the lemma.
\end{proof}

\subsection{Estimates of $\|\Na F\|_{L^p}$}

The following result shows that $|\Na F|_{\varphi}^2\in L^{\kappa}(\oo_{\varphi})$ for some $\kappa>n$ if $f\in L^{q}(\oo_{\varphi})$ for some $q>n.$

\begin{lem}\label{lem:F} Write $\td v=n+\Delta_g\varphi$. If $f\in L^{2b}(\oo_\varphi)$ for some $b\in (\frac n2, n)$, then for any $\kappa$ satisfying
\beq
0<\kappa<\frac {nb}{n-b},\label{eq:A16}
\eeq
 we have
\beq \int_M\; |\Na F|_{\varphi}^{2\kappa}\,\oo_{\varphi}^n\leq C(n, b, g, \kappa, \max_M|\eta|_{\oo_g}, \|F\|_{C^0}, \|f\|_{L^{2b}(\oo_{\varphi})}, \|\td v\|_{L^{2b(n-1)}(\oo_{\varphi})}).  \label{eq:B25}\eeq
Note that the upper bound of (\ref{eq:A16}) is greater than $n$.

\end{lem}
\begin{proof}
Let $w=e^{\frac 12 F}|\Na F|_{\varphi}^2$. By Chen-Cheng \cite[Equations (4.4)-(4.6)]{[CC1]}, we have
\beqn
\Delta_{\varphi}w&\geq& 2e^{\frac F2}\Na_{\varphi} F\cdot_{\varphi} \Na \Delta_{\varphi}F+\frac {e^{\frac F2}R_{j\bar i}F_i F_{\bar j}}{(1+\varphi_{i\bar i})(1+\varphi_{j\bar j})}+\frac 12 \Delta_{\varphi}F\,w\no\\
&\geq&2e^{\frac F2}\Na_{\varphi} F\cdot_{\varphi} \Na \Delta_{\varphi}F-C(g, \|F\|_{C^0}) \tr_{\varphi}g\, w-\frac 12(f+\max_M|\eta|_{\oo_g}\tr_{\varphi}g)\,w\no\\
&\geq&2e^{\frac F2}\Na_{\varphi} F\cdot_{\varphi} \Na \Delta_{\varphi}F-C(g, \|F\|_{C^0}, \max_M|\eta|_{\oo_g})\tr_{\varphi}g \,w-\frac 12 fw\no\\
&\geq&2e^{\frac F2}\Na_{\varphi} F\cdot_{\varphi} \Na \Delta_{\varphi}F-C(g, \|F\|_{C^0}, \max_M|\eta|_{\oo_g})\td v^{n-1} \,w-\frac 12 fw,\label{eq:A1}
\eeqn where $\td v=n+\Delta_{g}\varphi.$
All calculations with fractional powers of $w$ below are justified by replacing $w$ in the test functions by $w+\delta$, $\delta>0$, and then letting $\delta\downarrow0$. Multiplying both sides of (\ref{eq:A1}) by $w^{2p}$ and integrating by parts, for any $p>0$  we compute that
\beqn &&\int_M\; 2p w^{2p-1}|\Na w|_{\varphi}^2\,\oo_{\varphi}^n=\int_M\; -w^{2p}\Delta_{\varphi}w\,\oo_{\varphi}^n\no\\
&\leq& \int_M\;-2e^{\frac F2}w^{2p}\Na_\varphi F\cdot_\varphi \Na \Delta_{\varphi}F\,\oo_{\varphi}^n +\frac 12\int_M\; fw^{2p+1}\,\oo_{\varphi}^n\no\\
&&+\int_M\;C(g, \|F\|_{C^0}, \max_M|\eta|_{\oo_g})\td v^{n-1}w^{2p+1}\,\oo_{\varphi}^n\no\\
&\leq&\int_M\; \Big(p w^{2p-1}|\Na w|_{\varphi}^2+(4p+2)w^{2p}e^{\frac 12 F}(\Delta_{\varphi}F)^2
+w^{2p+1}|\Delta_{\varphi}F|\Big)\,\oo_{\varphi}^n\no\\
&&+\int_M\;C(g, \|F\|_{C^0}, \max_M|\eta|_{\oo_g})\td v^{n-1}w^{2p+1}\,\oo_{\varphi}^n+\frac 12\int_M\; fw^{2p+1}\,\oo_{\varphi}^n, \label{eq:A2a}
\eeqn where we used Chen-Cheng \cite[Equation (4.19)]{[CC1]}.
We note that
\beq
|\Delta_{\varphi}F|\leq |f|+|\tr_{\varphi}\eta|\leq |f|+C(\max_M|\eta|_{\oo_g}, \|F\|_{C^0})\td v^{n-1}. \label{eq:A3}
\eeq
Inserting the estimate (\ref{eq:A3}) into (\ref{eq:A2a}) and regrouping terms, we see that
\beqn
&&\int_M\; p w^{2p-1}|\Na w|_{\varphi}^2\,\oo_{\varphi}^n
\leq C(g, \|F\|_{C^0}, \max_M|\eta|_{\oo_g})\no\\
&&\cdot\int_M\; \Big((p+1) w^{2p}f^2+(p+1) w^{2p}\td v^{2n-2}
+w^{2p+1}|f|+\td v^{n-1}w^{2p+1}\Big)\,\oo_{\varphi}^n.\label{eq:A2b}
\eeqn
Then it also implies that
\beqn &&\int_M\; |\Na (w^{p+\frac 12})|_{\varphi}^2\,\oo_{\varphi}^n=(p+\frac 12)^2\int_M\;  w^{2p-1}|\Na w|_{\varphi}^2\,\oo_{\varphi}^n\no\\
&\leq&C(g, \|F\|_{C^0}, \max_M|\eta|_{\oo_g})(p+\frac 12)^2\no\\
&&\cdot\int_M\; \Big( \frac {p+1}{p}w^{2p}f^2+ \frac {p+1}{p}w^{2p}\td v^{2n-2}
+\frac 1pw^{2p+1}|f|+\frac 1p\td v^{n-1}w^{2p+1}\Big)\,\oo_{\varphi}^n.\label{eq:A4}
\eeqn
We write $z=w^{p+\frac 12}$ and apply Lemma \ref{lem:sob} to obtain that for any $\ga\in (1, \frac n{n-1})$
\beqn &&
\Big(\int_M\; z^{2\ga}\,\oo_{\varphi}^n\Big)^{\frac 1{\ga}}\leq C(n, \ga, g, \|F\|_{C^0})\int_M\; (|\Na z|_{\varphi}^2+z^2)\,\oo_{\varphi}^n  \no
\\&\leq &C(n, \ga, g, \|F\|_{C^0}, \max_M|\eta|_{\oo_g})(p+\frac 12)^2\no\\
&&\cdot\int_M\; \Big( \frac {p+1}{p} z^{\frac {4p}{2p+1}}f^2+ \frac {p+1}{p}z^{\frac {4p}{2p+1}}\td v^{2n-2}
+\frac 1p z^2|f|+\frac 1p \td v^{n-1}z^2+z^2\Big)\,\oo_{\varphi}^n. \label{eq:A8}
\eeqn
We  estimate each term of (\ref{eq:A8}).  Firstly, we note that the H\"older inequality implies that
\beq
I_1:=\int_M\;   z^{\frac {4p}{2p+1}}f^2\,\oo_{\varphi}^n
\leq\Big(\int_M\; z^{\frac {4pa}{2p+1}}\,\oo_{\varphi}^n\Big)^{\frac 1a} \Big(\int_M\; |f|^{2b}\,\oo_{\varphi}^n\Big)^{\frac 1b},
\eeq where $a, b>1$ satisfy $
\frac 1a+\frac 1b=1
$.  We assume that $b\in (\frac n2, n)$ and check that $
\frac {2b}{2b-1}<\frac n{n-1}<\frac {b}{b-1}=a.
$ Fix \beq \ga\in \Big(\frac {2b}{2b-1}, \frac n{n-1}\Big). \label{eq:ga}\eeq
If $p$ satisfies \beq \frac {b-1}{2(b+1)}< p< \frac {\ga}{2(a-\ga)},\label{eq:A15}\eeq where we used $b<n$ such that $\gamma<a$ in the last inequality,   we have \beq
1<q:=\frac {4pa}{2p+1}<2\ga<\frac {2n}{n-1}. \label{eq:A9}
\eeq
From Lemma \ref{lem:inter},
\beq
\|z\|_{L^q(\oo_{\varphi})}\leq \ee \|z\|_{L^{2\ga}(\oo_{\varphi})}+C(n, b, p, \ga, \ee)\|z\|_{L^1(\oo_{\varphi})}.\label{eq:A13}
\eeq where $\ee>0$ is small. Thus, using (\ref{eq:A13}) and the H\"older inequality we derive that
\beqn
I_1&\leq&\Big(\int_M\; z^{q}\,\oo_{\varphi}^n\Big)^{\frac 1a} \Big(\int_M\; |f|^{2b}\,\oo_{\varphi}^n\Big)^{\frac 1b}\no\\
&\leq&\Big(\ee \|z\|_{L^{2\ga}(\oo_{\varphi})}+C(n, b, p, \ga, \ee)\|z\|_{L^1(\oo_{\varphi})}\Big)^{\frac {4p}{2p+1}} \Big(\int_M\; |f|^{2b}\,\oo_{\varphi}^n\Big)^{\frac 1b}\no\\
&\leq&\frac {2p}{2p+1}\Big(\ee \|z\|_{L^{2\ga}(\oo_{\varphi})}+C(n, b, p, \ga, \ee)\|z\|_{L^1(\oo_{\varphi})}\Big)^2+\frac 1{2p+1}\Big(\int_M\; |f|^{2b}\,\oo_{\varphi}^n\Big)^{\frac {2p+1}b}\no\\
&\leq&\frac {4p}{2p+1}\Big(\ee^2 \|z\|_{L^{2\ga}(\oo_{\varphi})}^2+C(n, b, p, \ga, \ee)\|z\|_{L^1(\oo_{\varphi})}^2\Big)+\frac 1{2p+1}\Big(\int_M\; |f|^{2b}\,\oo_{\varphi}^n\Big)^{\frac {2p+1}b}. \label{eq:A12}
\eeqn
Secondly, we bound
 \beqn
I_2&:=&\int_M\; z^{\frac {4p}{2p+1}}\td v^{2n-2}\,\oo_{\varphi}^n\leq \Big(\int_M\; z^{\frac {4pa}{2p+1}}\,\oo_{\varphi}^n\Big)^{\frac 1a}\Big(\int_M\; \td v^{(2n-2)b}\,\oo_{\varphi}^n\Big)^{\frac 1b}\no\\
&=&\Big(\int_M\; z^{q}\,\oo_{\varphi}^n\Big)^{\frac 1a}\Big(\int_M\; \td v^{(2n-2)b}\,\oo_{\varphi}^n\Big)^{\frac 1b}\no\\
&\leq&\frac {4p}{2p+1}\Big(\ee^2 \|z\|_{L^{2\ga}(\oo_{\varphi})}^2+C(n,b,p,\ga,\ee)\|z\|_{L^1(\oo_{\varphi})}^2\Big)+\frac 1{2p+1}\Big(\int_M\; \td v^{(2n-2)b}\,\oo_{\varphi}^n\Big)^{\frac {2p+1}b}. \label{eq:A14}
\eeqn
Thirdly, we see that
\beqn
I_3&=&\int_M\; z^2|f|\,\oo_{\varphi}^n\leq \Big(\int_M\; z^{2s}\,\oo_{\varphi}^n \Big)^{\frac 1{s}}\Big(\int_M\;  |f|^{2b}\,\oo_{\varphi}^n\Big)^{\frac 1{2b}} \label{eq:B09}
\eeqn where  $s=\frac {2b}{2b-1}\in (1, \ga)$ by (\ref{eq:ga}). Lemma \ref{lem:inter} gives
 \beq
\|z\|_{L^{2s}(\oo_{\varphi})}\leq \ee \|z\|_{L^{2\ga}(\oo_{\varphi})}+C(n, b, \ga, \ee)\|z\|_{L^1(\oo_{\varphi})}.\label{eq:A10}
\eeq Taking (\ref{eq:A10}) in (\ref{eq:B09}), we have
\beqn
I_3&\leq& \Big( \ee \|z\|_{L^{2\ga}(\oo_{\varphi})}+C(n, b, \ga, \ee)\|z\|_{L^1(\oo_{\varphi})} \Big)^{2}\Big(\int_M\;  |f|^{2b}\,\oo_{\varphi}^n\Big)^{\frac 1{2b}}\no\\
&\leq&2\Big(\ee^2 \|z\|^2_{L^{2\ga}(\oo_{\varphi})}+C(n, b, \ga, \ee)\|z\|_{L^1(\oo_{\varphi})}^2 \Big)\Big(\int_M\;  |f|^{2b}\,\oo_{\varphi}^n\Big)^{\frac 1{2b}}. \label{eq:A91}
\eeqn
Lastly, we estimate the remaining terms of (\ref{eq:A8}).
\beqn
\int_M\; \td v^{n-1} z^2\,\oo_{\varphi}^n&\leq &\Big(\int_M\; \td v^{2b(n-1)}\,\oo_{\varphi}^n\Big)^{\frac 1{2b}}\Big(\int_M\; z^{2s}\,\oo_{\varphi}^n\Big)^{\frac 1s}, \no
\eeqn where $s=\frac {2b}{2b-1}$ as above.  By (\ref{eq:A10}), we have
\beqn
\int_M\; \td v^{n-1}z^2\,\oo_{\varphi}^n&\leq &\Big(\int_M\; \td v^{2b(n-1)}\,\oo_{\varphi}^n\Big)^{\frac 1{2b}}\Big(2\ee^2 \|z\|^2_{L^{2\ga}(\oo_{\varphi})}+C(n, b, p, \ga, \ee)\|z\|_{L^1(\oo_{\varphi})}^2  \Big).\label{eq:B24}
\eeqn
In conclusion, we insert (\ref{eq:A12}), (\ref{eq:A14}), (\ref{eq:A91}), (\ref{eq:B24}) into (\ref{eq:A8}), and obtain that, for any $p$ satisfying (\ref{eq:A15}),
\beqs
 \|z\|_{L^{2\ga}(\oo_{\varphi})}^2&\leq& C(n, \ga, b, g,  \|F\|_{C^0}, \|f\|_{L^{2b}(\oo_{\varphi})}, \|\td v\|_{L^{2b(n-1)}(\oo_{\varphi})}, \max_M|\eta|_{\oo_g})  (p+\frac 12)^2(p+1)p^{-1} \\
&&\cdot\Big(\ee^2 \|z\|_{L^{2\ga}(\oo_{\varphi})}^2+C(n, b, p, \ga, \ee)\|z\|_{L^1(\oo_{\varphi})}^2\\&&+\Big(\int_M\; |f|^{2b}\,\oo_{\varphi}^n\Big)^{\frac {2p+1}b}+\Big(\int_M\; \td v^{(2n-2)b}\,\oo_{\varphi}^n\Big)^{\frac {2p+1}b}\Big).
\eeqs
Therefore, we can choose $\ee$ small such that, for any $p$ satisfying (\ref{eq:A15}),
\beqn
 \|z\|_{L^{2\ga}(\oo_{\varphi})}&\leq&C(n, \ga, b, g, p,  \|F\|_{C^0}, \|f\|_{L^{2b}(\oo_{\varphi})}, \|\td v\|_{L^{2b(n-1)}(\oo_{\varphi})}, \max_M|\eta|_{\oo_g})\no\\
&&\cdot(\|z\|_{L^1(\oo_{\varphi})}+1).\label{eq:B05}
\eeqn
If $p>1/2$, by Lemma \ref{lem:inter}, we have
\beqn
\|z\|_{L^1(\oo_{\varphi})}&=&\|w\|_{L^{p+\frac 12}(\oo_{\varphi})}^{p+\frac 12}
\leq\Big(\ee'\|w\|_{L^{(p+\frac 12)2\ga}(\oo_{\varphi})}+C(n, p, \ga, \ee')\|w\|_{L^1(\oo_{\varphi})}\Big)^{p+\frac 12}\no\\
&\leq&C(p)\Big(\ee'^{p+\frac 12}\|w\|^{p+\frac 12}_{L^{(p+\frac 12)2\ga}(\oo_{\varphi})}+C(n, p, \ga, \ee')\|w\|^{p+\frac 12}_{L^1(\oo_{\varphi})}\Big)\no\\
&=&C(p)\Big(\ee'^{p+\frac 12}\|z\|_{L^{2\ga}(\oo_{\varphi})}+C(n, p, \ga, \ee')\|w\|^{p+\frac 12}_{L^1(\oo_{\varphi})}\Big),\label{eq:B06}
\eeqn If $0<p\leq1/2$, H\"older's inequality instead gives
\beq
\|z\|_{L^1(\oo_\varphi)}=\int_M w^{p+1/2}\,\oo_\varphi^n
\leq\left(\int_M\oo_g^n\right)^{1/2-p}\|w\|_{L^1(\oo_\varphi)}^{p+1/2}.
\label{eq:B06small}
\eeq In either case, we use
\beqn
\|w\|_{L^1(\oo_{\varphi})}&\leq&C(\|F\|_{C^0})\int_M\; |\Na F|_{\varphi}^2\,\oo_{\varphi}^n=C(\|F\|_{C^0})\int_M\; (-F\Delta_{\varphi}F)\,\oo_{\varphi}^n\no\\
&\leq&C(\|F\|_{C^0})\int_M\; (|f|+C(\max_M|\eta|_{\oo_g}, \|F\|_{C^0})\td v^{n-1})\,\oo_{\varphi}^n\no\\
&\leq&C(\max_M|\eta|_{\oo_g}, \|F\|_{C^0}, \|f\|_{L^{2b}(\oo_{\varphi})}, \|\td v\|_{L^{n-1}(\oo_{\varphi})}). \label{eq:B14}
\eeqn

Combining (\ref{eq:B05}) with (\ref{eq:B06}) or (\ref{eq:B06small}) and (\ref{eq:B14}), and absorbing the term with $\epsilon^{\prime}$ when $p>1/2$, we obtain
\beqn
\int_M\; w^{(2p+1)\ga}\,\oo_{\varphi}^n=\|z\|_{2\ga}^{2\ga}\leq C(n, \ga, b, g, p, \max_M|\eta|_{\oo_g}, \|F\|_{C^0}, \|f\|_{L^{2b}(\oo_{\varphi})}, \|\td v\|_{L^{2b(n-1)}(\oo_{\varphi})}). \label{eq:B07}
\eeqn
Let $\kappa=(2p+1)\ga.$ Since $p$ satisfies  (\ref{eq:A15}), we get
\beq
\frac {2b\ga}{b+1}<\kappa<\frac {b\ga}{b-(b-1)\ga}. \label{eq:B11}
\eeq
Since $\ga$ is any number satisfying (\ref{eq:ga}), $\frac {b\ga}{b-(b-1)\ga}$ can be chosen to be any number slightly less than $\frac {nb}{n-b}.$
Applying (\ref{eq:B11}) with the H\"older inequality to (\ref{eq:B07}), we obtain that for any $0<\kappa<\frac {nb}{n-b} $,
$ \|w\|_{L^{\kappa}(\oo_{\varphi})} $ is bounded.
  The lemma is proved.
\end{proof}

 \subsection{Estimates of $|\Na \varphi|$}
In this subsection, we will show $|\Na \varphi|$ is bounded when $f\in L^b(\oo_{\varphi})$ with $b>n$.

 \begin{lem}\label{lem:equ}
 Let
 \beqs
  u=e^A(|\nabla\varphi|^2+10),
  \quad A(F,\varphi)=-(F+\la\varphi)+\frac{1}{2}\varphi^2,
 \eeqs where $\la>0$ is sufficiently large, depending only on $\|\varphi\|_{C^0}, \max_M|\eta|_{\oo_g}$ and $g$.
 Then we have
 the inequality
\beq
\Delta_{\varphi}u\geq \td f u+\frac 1{n-1}|\Na \varphi|^{2+\frac 2n}e^{-\frac Fn}e^A, \label{eq:A31}
\eeq where $\td f=f-\lambda(n+2)+(n+2)\varphi.$

\end{lem}
\begin{proof} Utilising Chen-Cheng \cite[Equation (2.24)]{[CC1]}, we have
\beqn
&&e^{-A}\Delta_{\varphi}(e^A(|\nabla\varphi|^2+K))\no\\&\geq& K\frac{|-F_i-\lambda\varphi_i+\varphi\varphi_i|^2}{1+\varphi_{i\bar{i}}}+\frac{|\varphi_i|^2(|\nabla\varphi|^2+K)}{1+\varphi_{i\bar{i}}}\no\\
&& +\sum_i\frac{\lambda-\eta_{i\bar i}-\varphi}{1+\varphi_{i\bar{i}}}(|\nabla\varphi|^2+K)+\bigg(f-\lambda n+n\varphi\bigg)(|\nabla\varphi|^2+K)\no\\
&&  -C(g)|\nabla\varphi|^2\sum_i\frac{1}{1+\varphi_{i\bar{i}}}+
\frac{\varphi_{i\bar{i}}^2}{1+\varphi_{i\bar{i}}}+
(-2\lambda+2\varphi)|\nabla\varphi|^2\no\\&&
+2Re\bigg(\frac{(F_i+\lambda\varphi_i-\varphi\varphi_i)\varphi_{\bar{i}}}{1+\varphi_{i\bar{i}}}\bigg).\label{eq:A17}
\eeqn
where $C(g)$ depends only on $g$. Since
\beqn
\frac{|(F_i+\lambda\varphi_i-\varphi\varphi_i)\varphi_{\bar{i}}|}{1+\varphi_{i\bar{i}}}
&\leq&\frac{1}{2}\frac{|F_i+\lambda\varphi_i-\varphi\varphi_i|^2}{1+\varphi_{i\bar{i}}}+\frac{1}{2}\frac{|\varphi_i|^2}{1+\varphi_{i\bar{i}}}\no\\
&\leq&\frac{1}{2}\frac{|F_i+\lambda\varphi_i-\varphi\varphi_i|^2}{1+\varphi_{i\bar{i}}}+\frac{1}{2}|\nabla\varphi|^2\sum_i\frac{1}{1+\varphi_{i\bar{i}}}.\label{eq:A18}
\eeqn
Combining (\ref{eq:A17}) with (\ref{eq:A18}), we have
\beqn &&
e^{-A}\Delta_{\varphi}(e^A(|\nabla\varphi|^2+K))\no\\&\geq& \frac{|\varphi_i|^2(|\nabla\varphi|^2+K)}{1+\varphi_{i\bar{i}}} +\sum_i\frac{\lambda-\eta_{i\bar i}-\varphi}{1+\varphi_{i\bar{i}}}(|\nabla\varphi|^2+K) -C(g)|\nabla\varphi|^2\sum_i\frac{1}{1+\varphi_{i\bar{i}}}\no\\
&&+(f-\lambda n+n\varphi)(|\nabla\varphi|^2+K)+
(-2\lambda+2\varphi)|\nabla\varphi|^2\no\\
&\geq& \frac{|\varphi_i|^2(|\nabla\varphi|^2+K)}{1+\varphi_{i\bar{i}}}+
9|\nabla\varphi|^2\sum_i\frac{1}{1+\varphi_{i\bar{i}}}+(f-\lambda n+n\varphi)(|\nabla\varphi|^2+K)+(-2\lambda+2\varphi)|\nabla\varphi|^2\no\\
&\geq&\frac 1{n-1}|\Na \varphi|^{2+\frac 2n}e^{-\frac Fn}+(f-\lambda(n+2)+(n+2)\varphi)(|\nabla\varphi|^2+K),\label{eq:A19}
\eeqn where $K=10$ and we choose $\lambda\geq\|\varphi\|_{C^0}+\max_M|\eta|_{\oo_g}+C(g)+10$, with $C(g)$ enlarged if necessary. We used $(-2\lambda+2\varphi)|\nabla\varphi|^2\geq(-2\lambda+2\varphi)(|\nabla\varphi|^2+K)$. Note that in the last inequality of (\ref{eq:A19}), we used the inequality from the proof of \cite[Theorem 2.2]{[CC1]}
\beq
\sum_{i}\,\frac {|\varphi_i|^2|\Na \varphi|^2}{1+\varphi_{i\bar i}}+|\nabla\varphi|^2\sum_i\frac{1}{1+\varphi_{i\bar{i}}}\geq \frac 1{n-1}|\Na \varphi|^{2+\frac 2n}e^{-\frac Fn}. \label{eq:A20}
\eeq Combining (\ref{eq:A19}) with (\ref{eq:A20}), we have (\ref{eq:A31}).
The lemma is proved.
\end{proof}

\begin{lem}\label{lem:varphi2}If $f\in L^b(\oo_\varphi)$ for some $b>n$, then
\beq
|\Na \varphi|\leq C(n, b, g, \|f\|_{L^b(\oo_{\varphi})}, \|\varphi\|_{C^0}, \|F\|_{C^0}, \max_M|\eta|_{\oo_g}).
\eeq

\end{lem}
\begin{proof}Let $p\geq 2$. From Lemma \ref{lem:equ}, $u=e^A(|\nabla\varphi|^2+10)$ satisfies
\beq
\Delta_{\varphi}u\geq \td f u+h,
\quad h=\frac 1{n-1}|\Na \varphi|^{2+\frac 2n}e^{-\frac Fn}e^A.\label{eq:u}
\eeq
Multiplying both
sides by $u^{p-1}$ and integrating by parts, we calculate that
\beqs &&
\frac {4(p-1)}{p^2}\int_M\; |\Na (u^{\frac {p}2})|_{\varphi}^2\,\oo_{\varphi}^n =
(p-1)\int_M\; u^{p-2}|\Na u|_{\varphi}^2\,\oo_{\varphi}^n\\
&=&-\int_M\;u^{p-1}\Delta_{\varphi}u\,\oo_{\varphi}^n\leq-\int_M\;(\td f u^{p}+hu^{p-1})\,\oo_{\varphi}^n\leq\int_M\;|\td f| u^{p}\,\oo_{\varphi}^n.
\eeqs
Lemma \ref{lem:sob} and the H\"older inequality further imply that, for any $\ga\in (1, \frac n{n-1})$,
\beqn
\Big(\int_M\; u^{p\ga}\,\oo_{\varphi}^n\Big)^{\frac 1{\ga}}\no&\leq&C(n, \ga, g, \|F\|_{C^0})
 \int_M\; (|\Na (u^{\frac p2})|_{\varphi}^2+u^p)\,\oo_{\varphi}^n\no\\
&\leq&\frac {p^2 C(n, \ga, g, \|F\|_{C^0})}{p-1}\int_M\; (|\td f|+1)u^p\,\oo_{\varphi}^n\no\\
&\leq&\frac {p^2 C(n, \ga, g, \|F\|_{C^0})}{p-1}\Big(\int_M\; (|\td f|+1)^b\,\oo_{\varphi}^n\Big)^{\frac 1b}\Big(\int_M\; u^{ap}\,\oo_{\varphi}^n\Big)^{\frac 1a},
\eeqn where $b>n$ and $a=\frac b{b-1}\in (1, \frac n{n-1}).$
We choose $\ga\in (a, \frac n{n-1}).$ Letting $z=u^{\frac p2}$, we rewrite the inequality above and then apply the interpolation inequality to get
\beqn &&
\|z\|_{L^{2\ga }(\oo_{\varphi})}\no\\&\leq&C(n, \ga, g, \|f\|_{L^b(\oo_{\varphi})},
\|\varphi\|_{C^0}, \|F\|_{C^0}, \max_M|\eta|_{\oo_g} )p^{\frac 12} \|z\|_{L^{2a}(\oo_{\varphi})} \no\\
&\leq&C(n, \ga, g, \|f\|_{L^b(\oo_{\varphi})}, \|\varphi\|_{C^0}, \|F\|_{C^0}, \max_M|\eta|_{\oo_g})p^{\frac 12}\Big(\ee'\|z\|_{L^{2\ga}(\oo_{\varphi})}+C(n, b, \ga, \ee')\|z\|_{L^1(\oo_{\varphi})}\Big).\no
\eeqn Thus, we can choose $\ee'$ small such that
\beq
\|z\|_{L^{2\ga}(\oo_{\varphi})}\leq C(n, \ga, g, \|f\|_{L^b(\oo_{\varphi})}, \|\varphi\|_{C^0}, \|F\|_{C^0}, \max_M|\eta|_{\oo_g})\cdot p^{N}\cdot \|z\|_{L^1(\oo_{\varphi})}\label{eq:B27}
\eeq for some $N>0$ independent of $p$.

Rewriting this in terms of $u$, we see that the inequality (\ref{eq:B27}) implies that
\beqn
\|u\|_{L^{p\ga}(\oo_{\varphi})}&=&\|z\|_{2\ga}^{\frac 2p}\leq C(n, \ga, g, \|f\|_{L^b(\oo_{\varphi})}, \|\varphi\|_{C^0}, \|F\|_{C^0}, \max_M|\eta|_{\oo_g})^{\frac 2p}p^{\frac {2N}p}\|z\|_{1}^{\frac 2p}\no\\
&=&C(n, \ga, g, \|f\|_{L^b(\oo_{\varphi})}, \|\varphi\|_{C^0}, \|F\|_{C^0}, \max_M|\eta|_{\oo_g})^{\frac 2p}\cdot p^{\frac {2N}p}\cdot \|u\|_{\frac p2}. \label{eq:D01}
\eeqn
Note that
\beqs
\int_M\; |\Na \varphi|^2\,\oo_g^n&=&\int_M\; (-\varphi\Delta_g \varphi)\,\oo_g^n
=\int_M\; -\varphi(n+\Delta_g \varphi)\,\oo_g^n+n\int_M\; \varphi\,\oo_g^n\\
&\leq&  \|\varphi\|_{C^0}\int_M\; (n+\Delta_g \varphi)\,\oo_g^n+C(n, g, \|\varphi\|_{C^0})\\
&\leq&C(n, g, \|\varphi\|_{C^0}).
\eeqs
Thus, we have
\beqn
\|u\|_{L^1(\oo_{\varphi})}&\leq& C(g, \max_M|\eta|_{\oo_g}, \|F\|_{C^0}, \|\varphi\|_{C^0})\Big(\int_M\; |\Na \varphi|^2\,\oo_g^n+1\Big)\nonumber\\&\leq& C(n, g, \max_M|\eta|_{\oo_g}, \|F\|_{C^0}, \|\varphi\|_{C^0}).\label{eq:D02}
\eeqn
Together with (\ref{eq:D01}) and (\ref{eq:D02}), the standard Moser iteration implies that
  $\|u\|_{L^\infty}$ is bounded. The lemma is proved.
\end{proof}

\subsection{Estimates of $\|n+\Delta_g \varphi\|_{C^0}$}
In this subsection, we show that   $\|n+\Delta_g \varphi\|_{C^0}$ is bounded by a constant depending on $\||\Na F|_{\varphi}^2\|_{L^{\kappa}(\oo_{\varphi})}$. First, we show the following integral inequality of $v$, which allows $1<\al\leq p$ in place of $\al\geq p$ in Lemma \ref{lem:va}.

\begin{lem}\label{lem:v}Let
 \beq
 v=e^{-\alpha(F+\lambda\varphi)}(n+\Delta_g\varphi). \nonumber
 \eeq
Let $p\geq 2$ and $1<\al\leq p$. There exists a constant $C(g, \max_M|\eta|_{\oo_g})>0$ such that for $\la>C(g, \max_M|\eta|_{\oo_g})$ we have
\beqn &&
\frac {3(p-1)}{p^2}\int_M\;  |\Na_{\varphi} v^{\frac p2}|_{\varphi}^2\;\oo_{\varphi}^n\leq \int_M\; \Big(\td f+\frac {\al\la}{\al-1}+\frac 1n e^{-\frac Fn}|R_g| \Big)v^{p}\oo_{\varphi}^n\no\\
&&+2(p-\al)\int_M\;  v^{p}|\Na F|_\varphi^2\,\oo_\varphi^n+\frac {2\al^2\la^2(p-\al)}{(\al-1)^2}\int_M\; e^B v^{p-1}|\Na \varphi|_g^2\,\oo_g^n, \label{eq:B01}
\eeqn where $B=(1-\al)F-\al \la \varphi$ and $\td f=\alpha (\lambda n-f).$

\end{lem}
\begin{proof}By (\ref{eq:v4a}), $v$ satisfies the inequality
\beqn &&
\int_M\;(p-1) v^{p-2}|\Na_{\varphi} v|_{\varphi}^2\;\oo_{\varphi}^n\nonumber\\
&\leq&\int_M\; \Big(-A v^{p+\frac{1}{n-1}}
+\td f v^p
 -e^{-\alpha(F+\lambda\varphi)}\Delta_g  Fv^{p-1}+e^{-\alpha(F+\lambda\varphi)}R_gv^{p-1}\Big)\,\oo_{\varphi}^n\nonumber\\
 &:=&I_1+I_2+I_3+I_4. \label{eq:v4}
\eeqn
Recall (\ref{eq:v5a})
\beqn
I_3
&=&\int_M\; -\frac 1{1-\al}e^{(1-\al)F-\al \la \varphi}\Delta_g ((1-\al)F-\al\la \varphi) v^{p-1} \, \oo_g^n\nonumber\\&&-\int_M \frac {\al\la}{1-\al}e^{(1-\al)F-\al \la \varphi}\Delta_g \varphi v^{p-1} \, \oo_g^n:=J_1+J_2. \label{eq:v5}
\eeqn
Assume that $\al>1$ and let $B=(1-\al)F-\al \la \varphi$. Then we calculate
\beqn
J_1&=&\frac 1{\al-1}\int_M\; e^B\Delta_gB v^{p-1}\,\oo_g^n\nonumber\\
&=&-\frac 1{\al-1}\int_M\;\langle \Na B, \Na(e^B v^{p-1})\rangle_g\,\oo_g^n\nonumber\\
&=&-\frac 1{\al-1}\int_M\;\Big(e^B v^{p-1}|\Na B|^2+(p-1)e^Bv^{p-2}\langle \Na B, \Na v\rangle_g\Big)\,\oo_g^n.\label{eq:v6}
\eeqn
Note that
\beqn
|\langle \Na B, \Na v\rangle|&\leq& \frac 1{\al-1}v|\Na B|^2+\frac {\al-1}{4v}|\Na v|^2,\label{eq:v13alpha}\\
e^B|\Na v|^2&\leq& e^B|\Na v|_{\varphi}^2(n+\Delta_g \varphi)=e^F  |\Na v|_{\varphi}^2 v.
\label{eq:v13alpha2}
\eeqn
Putting (\ref{eq:v13alpha}) and (\ref{eq:v13alpha2}) in (\ref{eq:v6}), we have
\beqn
J_1
&\leq&\frac {p-\al}{(\al-1)^2}\int_M\; e^B v^{p-1}|\Na B|^2\,\oo_g^n+\frac {p-1}4\int_M\;  v^{p-2}|\Na v|_{\varphi}^2\,\oo_\varphi^n. \label{eq:v14}
\eeqn
Moreover, recall (\ref{eq:v7a})
\beq
J_2\leq\int_M\;\frac {\al\la}{\al-1} v^{p} \, \oo_{\varphi}^n. \label{eq:v7}
\eeq
Combining (\ref{eq:v5})-(\ref{eq:v7}), we have
\beqn
I_3&\leq&\frac {p-\al}{(\al-1)^2}\int_M\; e^B v^{p-1}|\Na B|^2\,\oo_g^n+\frac {p-1}4\int_M\;  v^{p-2}|\Na v|_{\varphi}^2\,\oo_\varphi^n\no\\
&&+\int_M\;\frac {\al\la}{\al-1} v^{p} \, \oo_{\varphi}^n.  \label{eq:v8a}
\eeqn
Taking these inequalities in (\ref{eq:v4}), we obtain that
\beqn &&
\frac {3(p-1)}4\int_M\; v^{p-2}|\Na_{\varphi} v|_{\varphi}^2\;\oo_{\varphi}^n\no\\&\leq& \int_M\; \Big(-A v^{p+\frac{1}{n-1}}
+\td f v^p\Big)\oo_{\varphi}^n+\frac {p-\al}{(\al-1)^2}\int_M\; e^B v^{p-1}|\Na B|^2\,\oo_g^n\no\\&&+\int_M\;\frac {\al\la}{\al-1} v^{p} \, \oo_{\varphi}^n+\int_M\;e^{-\alpha(F+\lambda\varphi)}R_gv^{p-1} \,\oo_{\varphi}^n.\label{eq:v9}
\eeqn
We rewrite the inequality (\ref{eq:v9}) as
\beqs &&
\frac {3(p-1)}{4}\int_M\;  v^{p-2}|\Na_{\varphi} v|_{\varphi}^2\;\oo_{\varphi}^n+\int_M\; A v^{p+\frac{1}{n-1}}
\,\oo_{\varphi}^n\\&\leq& \int_M\; \Big(\td f v^p+\frac {\al\la}{\al-1} v^{p}+e^{-\alpha(F+\lambda\varphi)}R_gv^{p-1} \Big)\oo_{\varphi}^n\\
&&+\frac {p-\al}{(\al-1)^2}\int_M\; e^B v^{p-1}|\Na B|^2\,\oo_g^n.
\eeqs
Thus, we can choose $\la$ as in (\ref{eq:A22}) such that
\beqn &&
\frac {3(p-1)}{p^2}\int_M\;  |\Na_{\varphi} v^{\frac p2}|_{\varphi}^2\;\oo_{\varphi}^n+\frac {\la \al}4\int_M\; e^{\frac 1{n-1}\al (F+\la \varphi)-\frac F{n-1}} v^{p+\frac{1}{n-1}}
\,\oo_{\varphi}^n \nonumber\\&\leq& \int_M\; \Big(\td f v^p+\frac {\al\la}{\al-1} v^{p}+e^{-\alpha(F+\lambda\varphi)}R_gv^{p-1} \Big)\oo_{\varphi}^n\no\\&&+\frac {p-\al}{(\al-1)^2}\int_M\; e^B v^{p-1}|\Na B|^2\,\oo_g^n\no\\
&\leq& \int_M\; \Big(\td f+\frac {\al\la}{\al-1}+\frac 1n e^{-\frac Fn}|R_g| \Big)v^{p}\oo_{\varphi}^n+\frac {p-\al}{(\al-1)^2}\int_M\; e^B v^{p-1}|\Na B|^2\,\oo_g^n,\label{eq:v11}
\eeqn where we used
\beq
e^{-\al(F+\la \varphi)}=\frac v{n+\Delta_g \varphi}\leq\frac 1n e^{-\frac Fn}v. \label{eq:v10}
\eeq
Note that
\beq
|\Na B|^2\leq 2(1-\al)^2|\Na F|^2+2\al^2\la^2|\Na \varphi|^2.\label{eq:v10 B}
\eeq
Thus, taking (\ref{eq:v10 B}) in (\ref{eq:v11}), we obtain
\beqn &&
\frac {3(p-1)}{p^2}\int_M\;  |\Na_{\varphi} v^{\frac p2}|_{\varphi}^2\;\oo_{\varphi}^n+\frac {\la \al}4\int_M\; e^{\frac 1{n-1}\al (F+\la \varphi)-\frac F{n-1}} v^{p+\frac{1}{n-1}}
\,\oo_{\varphi}^n\no\\&\leq& \int_M\; \Big(\td f+\frac {\al\la}{\al-1}+\frac 1n e^{-\frac Fn}|R_g| \Big)v^{p}\oo_{\varphi}^n\no\\
&&+2(p-\al)\int_M\; e^B v^{p-1}|\Na F|^2\,\oo_g^n+\frac {2\al^2\la^2(p-\al)}{(\al-1)^2}\int_M\; e^B v^{p-1}|\Na \varphi|^2\,\oo_g^n\no\\
&\leq&\int_M\; \Big(\td f+\frac {\al\la}{\al-1}+\frac 1n e^{-\frac Fn}|R_g| \Big)v^{p}\oo_{\varphi}^n\no\\
&&+2(p-\al)\int_M\;   v^{p}|\Na F|_\varphi^2\,\oo_\varphi^n+\frac {2\al^2\la^2(p-\al)}{(\al-1)^2}\int_M\; e^B v^{p-1}|\Na \varphi|^2\,\oo_g^n
\eeqn
The lemma is proved.
\end{proof}

Combining Lemma \ref{lem:varphi}, Lemma \ref{lem:F}, Lemma \ref{lem:varphi2} with Lemma \ref{lem:v}, we show that $n+\Delta_g\varphi$ is bounded from above.

\begin{lem}\label{lem:key} If $f\in L^b(\oo_{\varphi})$ for some $b>n$, then \beq
n+\Delta_g\varphi\leq C(n, b, g, \|f\|_{L^b(\oo_{\varphi})}, \|F\|_{C^0}, \|\varphi\|_{C^0}, \max_M|\eta|_{\oo_g}).
\eeq

\end{lem}
\begin{proof}Since $n+\Delta_g \varphi\geq n e^{\frac Fn}$, we compute
\beq
v^{p-1}=e^{\alpha(F+\lambda\varphi)}\frac {v^p}{n+\Delta_g\varphi}\leq \frac 1n e^{\alpha(F+\lambda\varphi)-\frac Fn}\,v^p.\label{eq:B08}
\eeq
Taking $z=v^{\frac p2}$ and $p\geq \al=2$ in  the inequality (\ref{eq:B01}), we see that
\beqn  &&
\frac {3(p-1)}{p^2}\int_M\;  |\Na z|_{\varphi}^2\;\oo_{\varphi}^n\leq \int_M\; \Big(\td f+2\la+\frac 1n e^{-\frac Fn}|R_g|\Big)z^2\oo_{\varphi}^n\no\\
&&+2p\int_M\;   z^{2}|\Na F|_\varphi^2\,\oo_\varphi^n+8\la^2p\int_M\; e^B v^{p-1}|\Na \varphi|^2\,\oo_g^n\no\\
&\leq&\int_M\; \Big(\td f+2\la+\frac 1n e^{-\frac Fn}|R_g|\Big)z^2\oo_{\varphi}^n\no\\
&&+2p\int_M\;   z^{2}|\Na F|_\varphi^2\,\oo_\varphi^n+C(n,b,g,\|F\|_{C^0},\|\varphi\|_{C^0},\|f\|_{L^b(\oo_\varphi)},\max_M|\eta|_{\oo_g})p\int_M\;  v^p\,\oo_\varphi^n,
\eeqn where we used (\ref{eq:B08}) and Lemma \ref{lem:varphi2} in the last inequality. Thus, we obtain that
\beq
\frac {3(p-1)}{p^2}\int_M\;  |\Na z|_{\varphi}^2\;\oo_{\varphi}^n\leq p\int_M\; Gz^2\,\oo_{\varphi}^n+
2p\int_M\;   z^{2}|\Na F|_\varphi^2\,\oo_\varphi^n
\eeq where
\beq
G=|\td f|+2\la+\frac1n e^{-F/n}|R_g|+C(n,b,g,\|F\|_{C^0},\|\varphi\|_{C^0},\|f\|_{L^b(\oo_\varphi)},\max_M|\eta|_{\oo_g})+1.
\eeq
Lemma \ref{lem:sob} implies that for any $\ga\in (1, \frac n{n-1})$
\beqn &&
\Big(\int_M\; z^{2\ga}\,\oo_{\varphi}^n\Big)^{\frac 1{\ga}}\no\\&\leq&C(n, \ga,  g, \|F\|_{C^0})\int_M\; (|\Na z|_{\varphi}^2+z^2)\,\oo_{\varphi}^n \no\\
&\leq&\frac {C(n, \ga, g, \|F\|_{C^0})p^3}{p-1}\int_M\; Gz^2\oo_{\varphi}^n+\frac {C(n, \ga, g, \|F\|_{C^0})p^3}{p-1}\int_M\;  z^{2}|\Na F|_\varphi^2\,\oo_\varphi^n.
\label{eq:B02}
\eeqn
By the assumption $b>n$, we have $s=\frac {b}{b-1}\in (1, \frac n{n-1})$. We will choose $\ga\in(s,n/(n-1))$ below, also taking $\kappa$ into account.
By the H\"older inequality and Lemma \ref{lem:inter}, we have
\beqn
\int_M\; Gz^2\oo_{\varphi}^n&\leq& \|G\|_{L^b(\oo_{\varphi})}\|z\|^2_{L^{2s}(\oo_{\varphi})}\no\\
&\leq&\|G\|_{L^b(\oo_{\varphi})}\Big(2\ee'^2 \|z\|^2_{L^{2\ga}(\oo_{\varphi})}+C(n, b, \ga,  \ee')\|z\|_{L^1(\oo_{\varphi})}^2 \Big). \label{eq:B03}
\eeqn      Similarly, we have
\beqn
\int_M\;  z^{2}|\Na F|_\varphi^2\,\oo_\varphi^n&\leq&\|\,|\Na F|_\varphi^2\,\|_{L^{\kappa}(\oo_{\varphi})}\Big(2\ee'^2 \|z\|^2_{L^{2\ga}(\oo_{\varphi})}+C(n, \kappa, \ga, \ee')\|z\|_{L^1(\oo_{\varphi})}^2 \Big).\label{eq:B04}
\eeqn
Here we assume $\kappa>n$ and we get $\frac {\kappa}{\kappa-1}\in (1, \frac n{n-1})$.  We choose $\ga$ satisfying
\beq
\max\Big\{\frac {\kappa}{\kappa-1}, \frac b{b-1}\Big\}<\ga<\frac n{n-1}.
\eeq
Inserting (\ref{eq:B03}), (\ref{eq:B04}) in (\ref{eq:B02}), we arrive at
\beqn
\|z\|_{L^{2\ga}(\oo_{\varphi})}^2&\leq& C(n, b, g, \|G\|_{L^b(\oo_{\varphi})},  \|\,|\Na F|_\varphi^2\,\|_{L^\kappa(\oo_{\varphi})},  \|F\|_{C^0} )p^3\no\\
&&\cdot\Big(\ee'^2\|z\|_{L^{2\ga}(\oo_{\varphi})}^2+C(n, b, \kappa, \ee')\|z\|_{L^1(\oo_{\varphi})}^2\Big).
\eeqn Thus,  we can choose $\ee'$ small such that
\beq
\|z\|_{L^{2\ga}(\oo_{\varphi})}^2\leq C(n, b, \kappa, g, \|G\|_{L^b(\oo_{\varphi})},  \|\,|\Na F|_\varphi^2\,\|_{L^\kappa(\oo_{\varphi})},  \|F\|_{C^0} ) p^N\|z\|_{L^1(\oo_{\varphi})}^2.
\eeq for some $N>0$ independent of $p.$ Thus, we conclude that
\beqs
\|v\|_{L^{p\ga}(\oo_{\varphi})}&=&\|z\|_{L^{2\ga}(\oo_{\varphi})}^{\frac 2p}\leq C(n, b, \kappa, g, \|G\|_{L^b(\oo_{\varphi})},  \|\,|\Na F|_\varphi^2\,\|_{L^\kappa(\oo_{\varphi})},  \|F\|_{C^0} )^{\frac 1p}p^{\frac Np}\|z\|_{L^1(\oo_{\varphi})}^{\frac 2p}\\&=&C(n, b, \kappa, g, \|G\|_{L^b(\oo_{\varphi})},  \|\,|\Na F|_\varphi^2\,\|_{L^\kappa(\oo_{\varphi})},  \|F\|_{C^0} )^{\frac 1p}p^{\frac Np}\|v\|_{L^{\frac p2}(\oo_{\varphi})}.
\eeqs
Choose $b_0\in(n/2,\min\{n,b/2\})$ and then $n<\kappa<nb_0/(n-b_0)$. Lemmas \ref{lem:F} and \ref{lem:varphi}, applied with exponent $2b_0<b$, bound $\|\,|\Na F|_\varphi^2\,\|_{L^\kappa(\oo_\varphi)}$. Fix $\ga$ as above and start the iteration with $p=2$, then replace $p$ successively by $2(2\ga)^j$. The product of the factors $C^{1/p}p^{N/p}$ is finite. Since $\alpha=2$ is fixed, (\ref{eq:B23}) bounds $\|v\|_{L^1(\oo_\varphi)}$. Hence $\|v\|_{L^\infty}$, and therefore $n+\Delta_g\varphi$, is uniformly bounded. This proves the lemma.
\end{proof}

\section{Proof of main theorems}

\subsection{Proof of Theorem \ref{theo:main1a}}

First, we note that
\beq
\int_M e^F\sqrt{1+F^2}\,\oo_g^n
\leq \int_M F e^F\,\oo_g^n+(1+2/e)\int_M\oo_g^n.\nonumber
\eeq
By the proof of Lu-Seyyedali \cite[Theorem 1.2]{[LS]}, with $f$ in place of scalar curvature, we have
\beq
\|F\|_{C^0}+\|\varphi\|_{C^0}\leq C.\nonumber
\eeq Here $C$ depends only on the data in the theorem, since their proof uses only $\max_M|\eta|_{\oo_g}$.
By Lemma \ref{lem:key}, $n+\Delta_g\varphi$ is uniformly bounded. Combining this with \eqref{MA F}, we obtain
\beq
\frac 1C\oo_g\leq\oo_\varphi\leq C\oo_g.
\eeq

Next, we show the H\"older estimate of $F$. In local holomorphic coordinates, the K\"ahler condition gives
\beq
\partial_i\big(\det(g_\varphi)g_\varphi^{i\bar j}\big)=0.\nonumber
\eeq Thus, \eqref{equation F} can be written as
\beq
\partial_i\big(\det(g_\varphi)g_\varphi^{i\bar j}\partial_{\bar j}F\big)
=\det(g_\varphi)\big(-f+\tr_\varphi\eta\big).\nonumber
\eeq
The coefficients are uniformly elliptic and the right hand side is bounded in $L^p$. Since $p>n$ and the real dimension is $2n$, the De Giorgi--Nash--Moser estimate implies that
\beq
\|F\|_{C^{\beta_0}}\leq C\nonumber
\eeq for some $\beta_0\in(0,1)$ depending only on the data in the theorem.

By \eqref{MA F} and Chen-Wang \cite[Theorem 1.5]{MR3464220}, the complex Hessian of $\varphi$ is uniformly bounded in $C^\beta$ for a fixed $0<\beta<\beta_0$. Applying the Schauder estimate to
\beq
\Delta_g\varphi=\tr_g(g_\varphi-g),\nonumber
\eeq we obtain $\|\varphi\|_{C^{2,\beta}}\leq C$. The coefficients in \eqref{equation F} are now uniformly bounded in $C^\beta$. By the elliptic $L^p$ estimate, we have
\beq
\|F\|_{W^{2,p}(\oo_g)}+\|\varphi\|_{C^{2,\beta}(\oo_g)}\leq C.\nonumber
\eeq
The theorem is proved.

\subsection{Proof of Theorem \ref{theo:main1}}
In order to prove Theorem \ref{theo:main1}, we recall some notation. The readers are referred to Chen-Cheng \cite{[CC2]} for details.  Let $(M, \oo_g)$ be a compact K\"ahler manifold of complex dimension $n$. We define the space of K\"ahler potentials
\beqn
\cH&=&\{\varphi\in C^{\infty}(M, \RR)\;|\; \oo_{\varphi}=\oo_g+\pbp\varphi>0\}, \label{eq:F01}\\
\cH_0&=&\{\varphi\in \cH\;|\;I_{\oo_g}(\varphi)=0\},\label{eq:F02}
\eeqn where the functional $I_{\oo_g}(\varphi)$ is defined by
\beq
I_{\oo_g}(\varphi)=\frac 1{(n+1)!}\int_M\; \varphi\sum_{k=0}^n\,\oo_g^k\wedge
\oo_{\varphi}^{n-k}.
\eeq
It is clear that for any path $\varphi(t)\in \cH$, we have
\beq
\frac d{dt}I_{\oo_g}(\varphi(t))=\frac 1{n!}\int_M\; \pd {\varphi(t)}t\oo_{\varphi(t)}^n. \label{eq:I}
\eeq
The $K$-energy is defined by
\beq
\cK(\varphi)=-\int_0^1\,dt\int_M\;\pd {\varphi_t}t(R(\oo_{\varphi_t})-\un R)\,\frac {\oo_{\varphi_t}^n}{n!}.
\eeq where $\varphi_t\in\cH$ is a path connecting $0$ and $\varphi$. Note that along the Calabi flow we have
\beq
\frac d{dt}\cK(\varphi(t))=-\int_M\;(R(\oo_{\varphi(t)})-\un R)^2\,\frac {\oo_{\varphi(t)}^n}{n!}\leq 0.
\eeq Therefore, the $K$-energy is non-increasing along the Calabi flow.
It is known that the $K$-energy can be written as
\beq
\cK(\varphi)=\int_M\; \log \frac {\oo_{\varphi}^n}{\oo_g^n}\,\frac {\oo_{\varphi}^n}{n!}+J_{-Ric(\oo_g)}(\varphi),
\eeq where for a smooth closed real $(1,1)$ form $\chi$, we define
 $$J_{\chi}(\varphi)=\int_0^1\,dt\int_M\;\pd {\varphi_t}t\Big(\chi\wedge \frac {\oo_{\varphi_t}^{n-1}}{(n-1)!}-\un \chi \frac {\oo_{\varphi_t}^n}{n!}\Big),$$ where $\varphi_t\in\cH$ is a path connecting $0$ and $\varphi$. Here
\beq
\un \chi=\frac {\int_M\;\chi\wedge \frac {\oo_g^{n-1}}{(n-1)!}}{\int_M\;\frac {\oo_g^n}{n!}}.
\eeq
We will use the following uniform version of He's short time construction.
\begin{lem}\label{lem:smoothing}
Fix $\beta\in(0,1)$ and $B,c>0$. The Calabi flows starting at smooth potentials $u$ satisfying $\|u\|_{C^{2,\beta}(\oo_g)}\leq B$ and $\oo_u\geq c\oo_g$ have a common existence time $\tau\in(0,1]$ and uniform $C^k$ bounds on $[\tau/2,\tau]$, for each $k$. The same holds for the modified flow with fixed extremal vector field.
\end{lem}
The proof is given in Appendix \ref{app:smoothing}.

\begin{proof}[Proof of Theorem \ref{theo:main1}]
First, since the scalar curvature equation can be written as
\beqs
(\oo_g+\pbp \varphi)^n&=&e^F\oo_g^n,\\
\Delta_{\varphi}F&=&-R(\oo_{\varphi(t)})+\tr_{\varphi}Ric(g),\quad \forall\; t\in [0, T).
\eeqs
By the assumption, $\|R(\oo_{\varphi(t)})\|_{L^p(\oo_{\varphi})}$ is uniformly bounded for $t\in [0, T)$. As a result of Theorem \ref{theo:main1a}, once one can show that there exists a constant $C$ depending on $T$ such that
\beq
\int_M\; F\,\oo_{\varphi}^n\leq C(T), \quad \forall\; t\in [0, T),\label{eq:F}
\eeq then Theorem \ref{theo:main1a} gives uniform metric equivalence and a uniform $C^{2,\beta}$ bound for $\varphi-\sup_M\varphi$, for some $\beta>0$. We may normalize $\varphi(0)\in\cH_0$; equation \eqref{eq:I} shows that this normalization is preserved. It bounds the additive constants, so $\|\varphi(t)\|_{C^{2,\beta}}$ is uniformly bounded. For every $0<\beta'<\beta$, these potentials belong to $c^{2,\beta'}$ and form a relatively compact family there. Lemma \ref{lem:smoothing} supplies a common existence time for the flow starting at each $\varphi(t)$. Starting sufficiently close to $T$ and using uniqueness extends the flow past $T$.

Next we bound $d_2(\varphi(0),\varphi(t))$. The Calabi energy is non-increasing, so the length of the flow segment gives
\beqn
d_2(\varphi(0),\varphi(t))
&\leq&\int_0^t\left(\int_M(R(\oo_{\varphi(s)})-\un R)^2\frac{\oo_{\varphi(s)}^n}{n!}\right)^{1/2}\,ds\no\\
&\leq&T\left(\int_M(R(\oo_{\varphi(0)})-\un R)^2\frac{\oo_{\varphi(0)}^n}{n!}\right)^{1/2}=:C(T),\quad t<T.
\label{eq:B12}
\eeqn

Next we bound $d_1$. For this discussion set $dV_\phi=\oo_\phi^n/n!$ and $\mathcal V=\int_M dV_g$. The path lengths defining $d_1$ and $d_2$ satisfy, for every smooth path $\phi_s$ joining $\phi_0$ to $\phi_1$,
\beq
L_1(\phi_\cdot):=\int_0^1\|\partial_s\phi_s\|_{L^1(dV_{\phi_s})}\,ds
\leq\mathcal V^{1/2}\int_0^1\|\partial_s\phi_s\|_{L^2(dV_{\phi_s})}\,ds
=\mathcal V^{1/2}L_2(\phi_\cdot).
\eeq
Taking the infimum with respect to all smooth paths connecting $\phi_0$ and $\phi_1$,  we have
\beq
d_1(\phi_0, \phi_1)\leq \mathcal V^{\frac12} d_2(\phi_0, \phi_1). \label{eq:B13}
\eeq Combining (\ref{eq:B12}) with (\ref{eq:B13}), we obtain that  $d_1(\varphi(0), \varphi(t))$ is uniformly bounded along the Calabi flow for $t\in [0, T)$.

Finally, we are ready to verify (\ref{eq:F}).
We assume that $\varphi(0)\in \cH_0$. By (\ref{eq:I}) we have
\beq
\frac d{dt}I_{\oo_g}(\varphi(t))=\frac 1{n!}\int_M\; \pd {\varphi(t)}t\,\oo_{\varphi(t)}^n=\frac 1{n!}\int_M\; (R(\oo_{\varphi(t)})-\un R)\,\oo_{\varphi(t)}^n=0.
\eeq
Thus, the solution of Calabi flow satisfies $\varphi(t)\in \cH_0$ for any $t\in [0, T).$ Since $d_1(\varphi(0), \varphi(t))$ is uniformly bounded and $\varphi(t)\in \cH_0$  for $t\in [0, T)$, by Chen-Cheng \cite[Lemma 4.4]{[CC2]} we have that $|J_{-Ric_g}(\varphi(t))|$ is uniformly bounded for $t\in [0, T)$.
 Since the $K$-energy $\cK(\varphi(t))$ is decreasing along the Calabi flow, we have
\beq
\frac1{n!}\int_M F\,\oo_\varphi^n=\cK(\varphi(t))-J_{-Ric_g}(\varphi(t))\leq\cK(\varphi(0))+|J_{-Ric_g}(\varphi(t))|\leq C(T).
\eeq
Here the constant $C$ depends on $T$ since $d_1(\varphi(0), \varphi(t))$ is bounded by a constant depending on $T$ by (\ref{eq:B12}).
Therefore, the theorem is proved.
\end{proof}

\subsection{Proof of Theorem \ref{theo:main2}}

First, we recall some preliminary results on extremal K\"ahler metrics. The readers are referred to He \cite{[He6]} for more details.
Futaki-Mabuchi \cite{[FM]} showed that there exists an extremal vector field determined by $(M, [\oo_g])$, regardless of whether an extremal K\"ahler metric exists or not. Moreover, such a real holomorphic vector field is unique up to the conjugate action of $\Aut_0(M).$ Here $\Aut_0(M)$ denotes the identity component of the automorphism group $\Aut(M)$.   We denote by $\Aut_0(M, V)$ the identity component of the subgroup of $\Aut_0(M)$ commuting with the flow of $V$, and we write $G=\Aut_0(M, V)$ for simplicity.

Given the (real) extremal vector field $V$, the corresponding holomorphic vector field $X=V-\sqrt{-1}JV$ is called the  (complex) extremal vector field. Assume that $\oo_g$ is invariant under the action of $JV$, which means that $L_{JV}\oo_g=0$. Then $X$ has a real potential function $\te_X$ such that $X=g^{j\bar k}\pd {\te_X}{\bar z_k}\pd {}{z_j}$ and we normalize $\te_X$ such that $\int_M\;\te_X\,\oo_g^n=0.$ We define $\cH_X$ as the space of K\"ahler potentials invariant under the action of $JV$:
\beq
\cH_X=\{\varphi\in \cH\;|\; L_{JV}\varphi=0\},\quad \cH_X^0=\cH_X\cap \cH_0,
\eeq where $\cH$ and $\cH_0$ are defined in (\ref{eq:F01}) and (\ref{eq:F02}) respectively.
For any $\varphi\in \cH_X$, the corresponding holomorphic potential with respect to the metric $\oo_{\varphi}$ is given by
\beq
\te_X(\varphi)=\te_X+V(\varphi).
\eeq It is known by Zhu \cite{[Zhu]} that $|\te_X(\varphi)|$ is uniformly bounded for any $\varphi\in\cH_X.$
For any $\varphi\in \cH_X$, the modified $K$-energy is defined by
\beq
\cK_X(\varphi)=-\int_0^1\,dt\int_M\; \pd {\varphi_t}t(R(\varphi_t)-\un R-\te_X(\varphi_t))\,\frac{\oo_{\varphi_t}^n}{n!},
\eeq where $\varphi_t\in \cH_X$ is a path connecting $0$ and $\varphi.$ The modified $K$-energy can be written as
\beq
\cK_X(\varphi)=\int_M\; \log\frac {\oo_{\varphi}^n}{\oo_g^n}\,\frac{\oo_{\varphi}^n}{n!}+J_{-Ric_g}(\varphi)+J^X(\varphi), \label{eq:K}
\eeq where $J^X$ is defined by
\beq
J^X(\varphi)=\int_0^1\,dt\int_M\; \pd {\varphi_t}t\,\te_X(\varphi_t)\,\frac{\oo_{\varphi_t}^n}{n!}.
\eeq Here $\varphi_t\in \cH_X$ is a path connecting $0$ and $\varphi.$

To state the properness result of W. Y. He, we introduce the distance function $d_{1, G}.$
Recall that Darvas \cite{[D]} introduced the $L^1$ norm on $T_{\varphi}\cH:$
\beq
\|\xi\|_{1, \varphi}=\int_M\; |\xi|\,\frac {\oo_{\varphi}^n}{n!},\quad \forall\; \xi\in T_{\varphi}\cH=C^{\infty}(M).
\eeq
The distance function $d_1(\varphi_0, \varphi_1)$ is the infimum of the length of the curves in $\cH$ connecting $\varphi_0$ and $\varphi_1.$ The $d_{1}$ distance relative to the action of $G$ is defined by
\beq
d_{1, G}(\varphi, \psi)=\inf_{\si\in G}d_1(\varphi, \si[\psi]),\quad \varphi,\psi\in\cH_X^0,
\eeq where $\si[\psi]\in \cH_0$ denotes the K\"ahler potential of $\si^*\oo_{\psi}.$

The properness of the $K$-energy on the K\"ahler-Einstein manifold is proved by Darvas-Rubinstein in \cite{[DR]}, and on the K\"ahler manifold which admits the constant scalar curvature K\"ahler metrics by Berman-Darvas-Lu in \cite{[BDL2]}. Here we reformulate the properness theorem proved by He \cite{[He6]} in the extremal K\"ahler metric case.
\begin{theo}\label{theo:He}(\cite[Theorem 3.1]{[He6]}) The class $[\oo_g]$ admits an extremal K\"ahler metric with extremal vector field $V$ if and only if the modified $K$-energy $\cK_X$ satisfies the following properness conditions:
\begin{enumerate}
       \item[(1).] There exist two constants $C>0$ and $D$ such that
\beq \cK_X(\varphi)\geq C d_{1, G}(0, \varphi)-D,\quad \forall\;\varphi\in \cH_X^0.\label{eq:F03}\eeq
       \item[(2).] $\cK_X$ is bounded below on $\cH_X$.
     \end{enumerate}

\end{theo}

\begin{proof}[Proof of Theorem \ref{theo:main2}] The proof consists of the following steps:

\emph{Step 1: Long time existence.} Due to Theorem \ref{theo:main1}, the Calabi flow exists for all time, under the assumption that the Calabi flow with the initial metric $\oo_{\varphi_0}$ has uniformly bounded $L^p$ scalar curvature. We denote by $\varphi(t)(t\in [0, \infty))$ the solution of Calabi flow with $\varphi(0)=\varphi_0\in\cH_X^0$, after adding a constant to $\varphi_0$.
Let $\{\si_t\}_{t\in\RR}$ with $\si_0=id$ be the one-parameter  family of automorphisms of $M$ generated by $V$. After normalization, the K\"ahler potential $\psi(t)(t\in [0, \infty))$ defined by $\oo_{\psi(t)}:=\si_{-t}^*\oo_{\varphi(t)}$ satisfies the modified Calabi flow equation
\beq
\pd {\psi(t)}t=R(\psi(t))-\un R-\te_X(\psi(t)),\quad \psi(0)=\varphi_0. \label{eq:C06}
\eeq Moreover, $\psi(t)\in\cH_X^0$, since $\int_M\te_X(\psi(t))\oo_{\psi(t)}^n=0$ and \eqref{eq:I} preserves the normalization.

\emph{Step 2: A priori estimates.} 
Since the modified $K$-energy is non-increasing along the modified Calabi flow, the modified $K$-energy is uniformly bounded from above along the flow.
By (\ref{eq:F03}), we have
\beq
d_{1, G}(0, \psi(t))\leq C, \quad \forall\; t>0,
\eeq where $C$ is a constant independent of $t$.
Thus, by the definition of $d_{1, G}$ for each $t>0$ we can find $\rho_t\in G$ and $\td \psi(t)\in \cH^0_X$ such that
\beq
\oo_{\td \psi(t)}=\rho_t^*\oo_{\psi(t)},\quad \sup_{t\in [0, \infty)}d_1(0, \td \psi(t))\leq C
\eeq for some $C>0. $ By Chen-Cheng \cite[Lemma 4.4]{[CC2]}, the functional $J_{-Ric_g}(\td \psi(t))$ is uniformly bounded
\beq
|J_{-Ric_g}(\td \psi(t))|\leq Cd_1(0, \td \psi(t))\leq C,\quad \forall\; t>0 \label{eq:C03}
\eeq
 and by He \cite[Proposition 2.2]{[He6]}, $J^X(\td \psi(t))$ is also uniformly bounded
\beq
|J^X(\td\psi(t))-J^X(0)|\leq Cd_{1}(\td \psi(t), 0)\leq C,\quad \forall\; t>0.  \label{eq:C04}
\eeq Here the constants are independent of $t$.
Since the modified $K$-energy is bounded below, it is $G$-invariant by \cite[Lemma 2.1]{[He6]}. Hence we have $\cK_X(\td \psi(t))=\cK_X(\psi(t))\leq C$ for all $t>0.$ Combining this with (\ref{eq:C03}), (\ref{eq:C04}) and (\ref{eq:K}), we obtain
\beq
\int_M\; \log \frac {\oo_{\td \psi(t)}^n}{\oo_g^n}\,\oo_{\td \psi(t)}^n\leq C, \quad \forall\; t>0. \label{eq:C05}
\eeq
Consider the equations of $(\td \psi(t), \td F)$
\beq
 (\oo_g+\pbp \td \psi)^n=e^{\td F} \oo_g^n,\quad
 \Delta_{\td \psi}\td F=-R(\oo_{\td \psi})+\tr_{\td \psi}Ric(\oo_g).\label{extremal tilde psi}
\eeq
Since the $L^p$ norm of scalar curvature is invariant under pullback, $\oo_{\td\psi(t)}=\rho_t^*\si_{-t}^*\oo_{\varphi(t)}$ has uniformly bounded $L^p$ scalar curvature for some $p>n$. Combining Theorem \ref{theo:main1a} with (\ref{eq:C05}), we obtain
\beq
C^{-1}\oo_g\leq\oo_{\td\psi(t)}\leq C\oo_g,\qquad
\|\td\psi(t)\|_{C^{2,\beta}(M)}\leq C
\label{eq:tilde-estimate}
\eeq
for some fixed $\beta\in(0,1)$, where $C$ is independent of $t$. Here we used $I_{\oo_g}(\td\psi(t))=0$ to control the additive constant.

\emph{Step 3: An extremal metric.}
By the assumption and Theorem \ref{theo:He}, there exists a smooth extremal K\"ahler metric with extremal vector field $V$. We denote its normalized potential by $\psi_e\in\cH_X^0$. Then
\beq
R(\oo_{\psi_e})-\un R-\te_X(\psi_e)=0.\no
\eeq Thus, $\psi_e$ is a stationary solution of the modified Calabi flow.

\emph{Step 4: Gauge fixing.}
First, we show that $d_2(\psi_e,\td\psi(t))$ is uniformly bounded. Let $\chi_s=(1-s)\psi_e+s\td\psi(t)$ for $s\in[0,1]$. By (\ref{eq:tilde-estimate}), we have
\beqn
d_2(\psi_e,\td\psi(t))
&\leq&\int_0^1\left(\int_M|\td\psi(t)-\psi_e|^2\frac{\oo_{\chi_s}^n}{n!}\right)^{1/2}\,ds\no\\
&\leq&\left(\int_M\frac{\oo_g^n}{n!}\right)^{1/2}\|\td\psi(t)-\psi_e\|_{C^0}\leq C.\no
\eeqn
The distance $d_2$ between two Calabi flows is non-increasing \cite{[St1],[BDL]}; cf. \cite[Theorem 1.5]{[Cal3]}. The same holds for two modified Calabi flows, since pulling back both flows by $\si_t$ gives two ordinary Calabi flows, and $\si_t$ acts isometrically on $\cH_0$. Since $\psi_e$ is stationary, we obtain
\beq
d_2(\psi_e,\psi(t))\leq d_2(\psi_e,\psi(0)),\quad \forall\;t\geq0.\no
\eeq
By the triangle inequality and the invariance of $d_2$ under $G$, we have
\beqn
d_2(\psi_e,\rho_t^{-1}[\psi_e])
&\leq&d_2(\psi_e,\psi(t))+d_2(\psi(t),\rho_t^{-1}[\psi_e])\no\\
&=&d_2(\psi_e,\psi(t))+d_2(\td\psi(t),\psi_e)\leq C.
\label{eq:orbit-bound}
\eeqn

Next, we use (\ref{eq:orbit-bound}) to control $\rho_t$. Let $K$ be the identity component of the holomorphic isometry group of $\oo_{\psi_e}$, and let $\mathfrak h_e$ be the space of real Hamiltonian holomorphic Killing fields of $\oo_{\psi_e}$. By Calabi's decomposition \cite[Equation (2.5)]{[He6]} and Darvas-Rubinstein \cite[Proposition 6.2]{[DR]}, we can write
\beq
\rho_t^{-1}=k_t\exp(JY_t),\qquad k_t\in K,\quad Y_t\in\mathfrak h_e.\no
\eeq
For $Y\in\mathfrak h_e$, let $h_Y$ be the function determined by
\beq
L_{JY}\oo_{\psi_e}=\pbp h_Y,\qquad
\int_M h_Y\frac{\oo_{\psi_e}^n}{n!}=0.\no
\eeq
The path $s\mapsto\exp(sJY)[\psi_e]$ is a Mabuchi geodesic with initial velocity $h_Y$. Since $k_t^*\oo_{\psi_e}=\oo_{\psi_e}$, it follows that
\beq
\left(\int_M h_{Y_t}^2\frac{\oo_{\psi_e}^n}{n!}\right)^{1/2}
=d_2(\psi_e,\exp(JY_t)[\psi_e])
=d_2(\psi_e,\rho_t^{-1}[\psi_e])\leq C.\no
\eeq
The map $Y\mapsto h_Y$ is linear and injective on $\mathfrak h_e$. Thus the left hand side defines a norm on this finite-dimensional vector space, and $Y_t$ is uniformly bounded. Since $K$ is compact, $\rho_t$ and $\rho_t^{-1}$ lie in a fixed compact subset of $G$.

Note that
\beq
\psi(t)=\rho_t^{-1}[0]+\td\psi(t)\circ\rho_t^{-1}+c_t,\no
\eeq where $c_t$ is determined by $I_{\oo_g}(\psi(t))=0$. The first term is uniformly bounded in every $C^k$ norm, and (\ref{eq:tilde-estimate}) bounds the second term in $C^{2,\beta}$. The normalization also gives $|c_t|\leq C$. Therefore,
\beq
\|\psi(t)\|_{C^{2,\beta}(M)}\leq C,\qquad
C^{-1}\oo_g\leq\oo_{\psi(t)}\leq C\oo_g,
\label{eq:D03}
\eeq where $C$ is independent of $t$.

\emph{Step 5: Convergence.}
By (\ref{eq:D03}) and Lemma \ref{lem:smoothing}, applied with initial time $t_0\geq0$, we obtain
\beq
\|\psi(t)\|_{C^k(M)}\leq C_k,\quad \forall\;t\geq1,\quad k\geq0.\no
\eeq
Write
\beq
r(t)=R(\oo_{\psi(t)})-\un R-\te_X(\psi(t)),\qquad
Ca(\psi(t))=\int_M r(t)^2\,\frac{\oo_{\psi(t)}^n}{n!}.
\eeq
Since the modified $K$-energy is bounded below, we have
\beqn
\frac d{dt}\cK_X(\psi(t))&=&-Ca(\psi(t)),\no\\
\int_0^\infty Ca(\psi(t))\,dt
&=&\cK_X(\psi(0))-\lim_{t\to\infty}\cK_X(\psi(t))<\infty.
\eeqn
Set $\mathcal D_t=\bar\partial\nabla_{\psi(t)}^{1,0}$. Since $\psi(t)$ is $JV$-invariant, the calculation in Huang-Zheng \cite[Section 3]{[HZ]} gives
\beq
\frac d{dt}Ca(\psi(t))=-2\int_M|\mathcal D_t r(t)|^2\,\frac{\oo_{\psi(t)}^n}{n!}\leq0.
\label{eq:modified-energy-decay}
\eeq
Thus, $Ca(\psi(t))\to0$ as $t\to\infty$. By the uniform $C^k$ estimates, there exists a sequence $t_i\to\infty$ such that $\psi(t_i)$ converges smoothly to $\psi_\infty\in\cH_X^0$. It follows that
\beq
R(\oo_{\psi_\infty})-\un R-\te_X(\psi_\infty)=0.\no
\eeq
Hence $\oo_{\psi_\infty}$ is an extremal K\"ahler metric and $\psi_\infty$ is stationary under the modified Calabi flow. The distance $d_2(\psi(t),\psi_\infty)$ is non-increasing and tends to zero along $t_i$. Therefore it tends to zero as $t\to\infty$. Every smoothly convergent subsequence must have the same limit $\psi_\infty$, so the uniform $C^k$ estimates imply that $\psi(t)\to\psi_\infty$ smoothly.

Finally, we show exponential convergence by the argument of Huang-Zheng \cite[Section 3]{[HZ]}. For a holomorphic vector field $Z$ with a holomorphy potential and $[Z,X]=0$, denote its mean-zero potential with respect to $\oo_{\psi(t)}$ by $\te_Z(t)$. The vanishing of the modified Futaki invariant gives
\beq
\int_M r(t)\frac{\oo_{\psi(t)}^n}{n!}=0,\qquad
\int_M r(t)\te_Z(t)\frac{\oo_{\psi(t)}^n}{n!}=0.\no
\eeq
We claim that there exists a constant $\lambda>0$ such that
\beq
\int_M|\mathcal D_t r(t)|^2\,\frac{\oo_{\psi(t)}^n}{n!}
\geq\lambda\int_M r(t)^2\,\frac{\oo_{\psi(t)}^n}{n!}
\label{eq:spectral-gap}
\eeq
for all sufficiently large $t$. Otherwise, there exist $t_j\to\infty$ with $Ca(\psi(t_j))>0$ such that, setting $u_j=r(t_j)/Ca(\psi(t_j))^{1/2}$, we have
\beq
\int_M u_j^2\frac{\oo_{\psi(t_j)}^n}{n!}=1,\qquad
\int_M|\mathcal D_{t_j}u_j|^2\frac{\oo_{\psi(t_j)}^n}{n!}\to0.\no
\eeq
The elliptic estimate for $\mathcal D_{t_j}$ and the smooth convergence of the metrics give a uniform $W^{2,2}(\oo_g)$ bound for $u_j$. Passing to a subsequence, $u_j$ converges weakly in $W^{2,2}$ and strongly in $L^2$ to a function $u_\infty$ satisfying
\beq
\mathcal D_\infty u_\infty=0,\qquad
\int_M u_\infty\frac{\oo_{\psi_\infty}^n}{n!}=0,\qquad
\int_M u_\infty^2\frac{\oo_{\psi_\infty}^n}{n!}=1.\no
\eeq
By elliptic regularity, $u_\infty$ is smooth and $Z=\nabla_{\psi_\infty}^{1,0}u_\infty$ is holomorphic. Moreover, $u_\infty$ is $JV$-invariant. Since $JV$ also preserves $\oo_{\psi_\infty}$, we have $[JV,Z]=0$, and hence $[X,Z]=0$.
For this fixed $Z$, its normalized potential at time $t_j$ is
\beq
\te_Z(t_j)=u_\infty+Z\big(\psi(t_j)-\psi_\infty\big)-a_j,\no
\eeq where $a_j$ is the mean of the first two terms with respect to $\oo_{\psi(t_j)}^n$. Thus $a_j\to0$ and $\te_Z(t_j)\to u_\infty$ smoothly. Passing to the limit in the orthogonality relation, we obtain
\beq
0=\lim_{j\to\infty}\int_M u_j\te_Z(t_j)\frac{\oo_{\psi(t_j)}^n}{n!}
=\int_M u_\infty^2\frac{\oo_{\psi_\infty}^n}{n!}=1,\no
\eeq which is a contradiction. This proves (\ref{eq:spectral-gap}).

Combining (\ref{eq:modified-energy-decay}) with (\ref{eq:spectral-gap}), we have
\beq
\frac d{dt}Ca(\psi(t))\leq-2\lambda Ca(\psi(t)),\qquad
Ca(\psi(t))\leq Ce^{-2\lambda t}.\no
\eeq
The uniform higher order estimates and Sobolev interpolation then give $\|r(t)\|_{C^k(M)}\leq C_k e^{-\lambda t/2}$ for every $k\geq0$. Integrating (\ref{eq:C06}) from $t$ to infinity, we obtain
\beq
\|\psi(t)-\psi_\infty\|_{C^k(M)}
\leq\int_t^\infty\|r(s)\|_{C^k(M)}\,ds
\leq C_k e^{-\lambda t/2}.\no
\eeq
Therefore, $\psi(t)$ converges exponentially fast to $\psi_\infty$ in every $C^k$ norm. The theorem is proved.

\end{proof}


\appendix
\renewcommand{\theequation}{\thesection.\arabic{equation}}
\section{Proof of Lemma \ref{lem:smoothing}}\label{app:smoothing}

\begin{proof}[Proof of Lemma \ref{lem:smoothing}]
We follow the proof of He \cite[Theorem 2.1]{[He4]}. All norms and covariant derivatives below are taken with respect to $g$. First, fix $0<\beta'<\beta$. Recall that $c^{2,\beta'}(M)$ is the closure of $C^\infty(M)$ in $C^{2,\beta'}(M)$. For $u\in C^{2,\beta}(M)$, let $u_\epsilon$ be smooth approximations which converge to $u$ in $C^2$ and are uniformly bounded in $C^{2,\beta}$. By interpolation, we have
\[
\|u_\epsilon-u\|_{C^{2,\beta'}}
\leq C\|u_\epsilon-u\|_{C^2}^{1-\beta'/\beta}
\|u_\epsilon-u\|_{C^{2,\beta}}^{\beta'/\beta}
\longrightarrow0.
\]
Thus, $C^{2,\beta}(M)\subset c^{2,\beta'}(M)$. Let $\mathcal A$ be the closure in $c^{2,\beta'}$ of the initial potentials in the lemma. The compact embedding shows that $\mathcal A$ is compact, and $\oo_u\geq c\oo_g$ for every $u\in\mathcal A$.

Let $S_\delta$ be smooth regularization operators on $M$. We write $v=S_\delta u$ and $h=g_v$. For $\delta_0>0$ sufficiently small, we have
\[
 h\geq \frac c2g,\qquad
 \sup_{u\in\mathcal A,\,0<\delta\leq\delta_0}
 \|S_\delta u\|_{C^{2,\beta}}<\infty,\qquad
 \sup_{u\in\mathcal A}\|S_\delta u-u\|_{c^{2,\beta'}}\longrightarrow0.
\]
Here we used the uniform $C^{2,\beta}$ bound and the interpolation inequality above. It follows that the metrics $h$, with $u\in\mathcal A$ and $0<\delta\leq\delta_0$, have compact closure in $c^{\beta'}$, and every limit is positive.

Next, we consider the linear equation. Set $E_0=c^{\beta'}$, $E_1=c^{4,\beta'}$, and $E_{1/2}=c^{2,\beta'}$. Let $P_h$ be the contraction of the symmetrized fourth covariant derivative with $h^{-1}\otimes h^{-1}$, normalized so that its principal part in holomorphic coordinates is
\[
 h^{i\bar j}h^{k\bar l}\partial_i\partial_{\bar j}
 \partial_k\partial_{\bar l}.
\]
Note that the coefficients of $P_h$ depend on $h$ and the connection of $g$, without involving derivatives of $h$. Moreover, its principal symbol is uniformly comparable to $|\xi|_g^4$.

For $0<T\leq1$, let $X_T$ and $Y_T$ be the weighted spaces with norms
\[
\begin{split}
 \|w\|_{X_T}&=\sup_{0\leq t\leq T}\|w(t)\|_{E_{1/2}}
 +\sup_{0<t\leq T}t^{1/2}
   \big(\|\partial_t w(t)\|_{E_0}+\|w(t)\|_{E_1}\big),\\
 \|q\|_{Y_T}&=\sup_{0<t\leq T}t^{1/2}\|q(t)\|_{E_0}.
\end{split}
\]
Here $w\in C([0,T],E_{1/2})$, and the weighted functions are continuous for $t>0$ and tend to zero as $t\to0$. According to \cite[Theorem 3.6]{[ShaoSimonett]}, the linear equation has a unique solution in these spaces, and
\begin{equation}\label{eq:He-linear}
 \|w\|_{X_T}\leq L\big(\|w(0)\|_{E_{1/2}}
                   +\|\partial_t w+P_hw\|_{Y_T}\big).
\end{equation}
We claim that $L$ can be chosen independently of $h$ and $T\leq1$. Indeed, the operators $P_h$ have compact closure in $\mathcal L(E_1,E_0)$, and every limit $P_*$ is strongly elliptic with $c^{\beta'}$ coefficients. If $P_h$ is sufficiently close to $P_*$, the term $(P_h-P_*)w$ can be absorbed in the estimate for $P_*$. By a finite covering, we obtain a uniform constant on $[0,1]$. Extending the weighted forcing from $[0,T]$ to $[0,1]$ with bounded norm gives the same estimate for $T\leq1$. In particular, $L$ does not depend on higher derivatives of $h$.

Now we write the evolving potential as $v+w$, where $v=S_\delta u$ is fixed. The Calabi flow equation becomes
\[
 \partial_t w+P_hw=N_h(w),\qquad
 N_h(w):=R(g_{v+w})-\un R+P_hw.
\]
By the expression of scalar curvature, the terms of highest order in $N_h(w)$ are of the form
\[
 \big(h^{-1}\otimes h^{-1}
       -g_{v+w}^{-1}\otimes g_{v+w}^{-1}\big)*\Na_g^4w,
 \qquad A(h,\Na_g^2w)*(\Na_g^3w)*(\Na_g^3w),
\]
where $*$ denotes contraction and $A$ is smooth as long as the metric is positive. The remaining terms are at most linear in $\Na_g^3w$ and involve derivatives of $h$ of order at most two. By interpolation, we have
\[
 \sup_{0<t\leq T}t^{1/4}\|w(t)\|_{c^{3,\beta'}}
 \leq C\left(\sup_{0\leq t\leq T}\|w(t)\|_{c^{2,\beta'}}\right)^{1/2}
 \left(\sup_{0<t\leq T}t^{1/2}\|w(t)\|_{c^{4,\beta'}}\right)^{1/2}.
\]
Thus, using the product estimates in $c^{\beta'}$, for $\|w\|_{X_T},\|z\|_{X_T}\leq r\leq1$ with $r$ sufficiently small, we obtain
\begin{equation}\label{eq:He-nonlinear}
\begin{split}
 \|N_h(w)\|_{Y_T}
 &\leq C_0r^2+D_h(T^{1/4}r+T^{1/2}),\\
 \|N_h(w)-N_h(z)\|_{Y_T}
 &\leq (C_0r+D_hT^{1/4})\|w-z\|_{X_T}.
\end{split}
\end{equation}
Here $C_0$ depends only on $g$, the uniform $c^{\beta'}$ bound of $h$, and its positive lower bound, while $D_h$ may also depend on $\|h\|_{C^{2,\beta'}}$. Note that the terms linear in $\Na_g^3w$ are bounded by $D_hT^{1/4}\|w\|_{X_T}$, and the term independent of $w$ is bounded by $D_hT^{1/2}$.

Finally, we choose $0<r\leq1$ such that $LC_0r\leq1/8$ and
$g_{v+w}\geq(c/4)g$ whenever $\|w\|_{c^{2,\beta'}}\leq r$, and we take
$\varepsilon=r/(4L)$. Since $L$ and $C_0$ are uniform, $r$ and $\varepsilon$ are independent of $\delta$. We then fix $\delta\in(0,\delta_0]$ sufficiently small such that
\[
 \sup_{u\in\mathcal A}\|u-S_\delta u\|_{c^{2,\beta'}}
 \leq\varepsilon.
\]
For this fixed $\delta$, all derivatives of $S_\delta u$ are uniformly bounded. Therefore, $D:=\sup_{u\in\mathcal A}D_{g_{S_\delta u}}<\infty$. We choose $\tau\in(0,1]$ such that
\[
 LD(\tau^{1/4}r+\tau^{1/2})\leq r/4,
 \qquad LD\tau^{1/4}\leq1/4.
\]
Let $x=u-S_\delta u$. For $w\in X_\tau$, let $\widetilde w$ solve $\partial_t\widetilde w+P_h\widetilde w=N_h(w)$ with $\widetilde w(0)=x$. Combining (\ref{eq:He-linear}) and (\ref{eq:He-nonlinear}), we see that $w\mapsto\widetilde w$ maps the closed ball $\{w\in X_\tau:w(0)=x,\ \|w\|_{X_\tau}\leq r\}$ into itself, and its contraction factor is at most $3/8$. This ball is nonempty since the homogeneous linear solution has norm at most $L\varepsilon=r/4$. By the contraction mapping theorem, there exists a unique solution on $[0,\tau]$. Here $\tau$ is independent of $u$.

Moreover, the $X_\tau$ bound gives a uniform $C^{4,\beta'}$ bound away from $t=0$, and the metrics are uniformly positive. Differentiating the equation and applying the interior parabolic Schauder estimates, we obtain uniform $C^k$ bounds on $[\tau/2,\tau]$ for every $k$. By uniqueness, these solutions agree with the smooth flows from the given initial data. The same conclusion holds for the modified flow by pulling back under the fixed one-parameter group, whose derivatives are bounded on $[0,\tau]$. Thus, the lemma is proved.
\end{proof}
\bigskip

\noindent{\bf Acknowledgements}
The authors would like to thank Professors Xiuxiong Chen, Bing Wang and Weiyong He for helpful discussions.  The third author would like to express his deepest gratitude to IHES and the K.C. Wong Education Foundation, also the CRM and the Simons Foundation, when parts of the work were undertaken during his stays.\\

\noindent{\bf Funding} Haozhao Li was partially  supported by NSFC grant  No. 12471058 and a research fund from  University of Science and Technology of China. Kai Zheng was partially  supported by NSFC grant No. 12571091. \\

\noindent\textbf{Data availability} Data sharing is not applicable to this article as no data sets were generated or analyzed during the current study.\\

\noindent\textbf{Declarations}\\

\noindent\textbf{Conflict of interest:} On behalf of all authors, the corresponding author states that there is no conflict of interest. \\


\begin{thebibliography}{10}







\bibitem{[BDL]}R. Berman, T. Darvas and C.H. Lu,
\emph{Convexity of the extended K-energy and the large time behavior of the weak Calabi flow}. {  Geom. Topol.} 21(2017), no.5, 2945-2988.
\bibitem{[BDL2]}R. Berman, T. Darvas, C. H. Lu,
\emph{Regularity of weak minimizers of the K-energy and applications to properness and K-stability},
Ann. Sci. Ec. Norm. Super. 53 (2020), no. 2, 267-289.


\bibitem{[Cal1]}E. Calabi,  \emph{Extremal K\"ahler metrics}.
 Seminar on Differential Geometry, pp. 259-290, Ann. of Math. Stud.,
 102, Princeton Univ. Press, Princeton, N.J., 1982.

\bibitem{[Cal3]}E. Calabi and X. X. Chen, \emph{The space of K\"ahler metrics}. II, J. Differential Geom. 61 (2002),
no. 2, 173-193.





\bibitem{[Chen]} X. X. Chen,  \emph{Calabi flow in Riemann surfaces revisited: a new point of view}. {  Internat. Math. Res. Notices} 2001, no. 6, 275-297.

\bibitem{[Chen2]}X. X. Chen, \emph{On the existence of constant scalar curvature K\"ahler metric: a new perspective}, Ann. Math. Qu\'e 42 (2018), no. 2, 169-189.


\bibitem{[CC1]}X. X. Chen and J. R. Cheng, \emph{On the constant scalar curvature K\"ahler metrics (I)-A priori estimates},
J. Amer. Math. Soc. 34 (2021), no. 4, 909-936.


\bibitem{[CC2]}X. X. Chen and J. R. Cheng, \emph{On the constant scalar curvature K\"ahler metrics (II)-Existence results},
J. Amer. Math. Soc. 34 (2021), no. 4, 937-1009.


\bibitem{[ChenHe1]}X. X. Chen and W. Y.  He,  \emph{On the Calabi flow}. { Amer. J. Math.} 130 (2008), no. 2, 539-570.
\bibitem{[ChenHe2]}X. X. Chen and W. Y.  He,  \emph{ The Calabi flow on K\"ahler surfaces with bounded Sobolev constant (I)}. { Math. Ann.} 354 (2012), no. 1, 227-261.
\bibitem{[ChenHe3]}X. X. Chen and W. Y.  He,  \emph{The Calabi flow on toric Fano surfaces}. { Math. Res. Lett.} 17 (2010), no. 2, 231-241.

\bibitem{[ChenHe4]}X. X. Chen and W. Y.  He, \emph{The complex Monge-Ampere equation on compact K\"ahler manifolds}.
Math. Ann. 354 (2012), no. 4, 1583-1600.







\bibitem{[Wang3]}X. X. Chen and B. Wang,   \emph{On the conditions to extend Ricci
flow (III)}. { Int. Math. Res. Not.} IMRN 2013,  No. 10, pp. 2349-2367.


\bibitem{MR3464220}
X. X. Chen and Y. Wang,
\emph{$C^{2,\alpha}$-estimate for Monge-Amp\`ere equations with
   H\"{o}lder-continuous right hand side}.
{Ann. Global Anal. Geom.} 49 (2016)
   no. {2}, 195-204.

\bibitem{[Chru]}P. T. Chru\'sciel, \emph{Semi-global existence and convergence of solutions of the Robinson-Trautman (2-dimensional Calabi) equation}. { Comm. Math. Phys.} 137 (1991), 289-313.

\bibitem{[D]}T. Darvas, \emph{The Mabuchi geometry of finite energy classes}. Adv. Math. 285 (2015), 182-219.

\bibitem{[DR]}T. Darvas, Y. A. Rubinstein, \emph{Tian's properness conjectures and Finsler geometry of the space of K\"ahler metrics}. J. Amer. Math. Soc. 30 (2017), no. 2, 347-387.

\bibitem{MR2103718} S. Donaldson,
   \emph{Conjectures in K\"{a}hler geometry}. Strings and geometry, Clay Math. Proc. 3, Amer. Math. Soc., Providence, RI, 2004, 71-78.




\bibitem{[FM]} Futaki, A., Mabuchi, T., \emph{Bilinear forms and extremal K\"ahler vector fields associated with K\"ahler classes},
Math. Ann. 301 (1995), no. 2, 199-210.

\bibitem{[GPSS]}B. Guo, D. H. Phong, J. Song, and J. Sturm, {\it Sobolev inequalities on K\"ahler spaces}, arXiv:2311.00221.

\bibitem{[GPSSs2]}
B. Guo, D. H. Phong, J. Song and J. Sturm, \emph{Diameter estimates in K\"ahler geometry II: removing the small degeneracy assumption}, Math. Z. \textbf{308} (2024), no.~3, Paper No.~43, 7 pp.



\bibitem{[GT2]}D. Gilbarg and N. S. Trudinger,   \emph{Elliptic partial differential equations of second order},
Classics Math.
Springer-Verlag, Berlin, 2001, xiv+517 pp.

\bibitem{[GT]}V. Guedj, T. D. T\^o, {\it K\"ahler families of Green's functions}, J. \'{E}c. polytech. Math. 12 (2025), 319-339.



\bibitem{[He4]} W. Y. He, \emph{Local solution and extension to the Calabi flow}. { J. Geom. Anal.} 23 (2013), no. 1,
270-282.
\bibitem{[He5]} W. Y. He, \emph{On the convergence of the Calabi
flow}. Proc. Amer. Math. Soc. 143 (2015), no. 3, 1273-1281.

\bibitem{[He6]} W. Y. He, \emph{On Calabi's extremal metric and properness}, Trans. Amer. Math. Soc. 372 (2019), 5595-5619.

\bibitem{[Huang]}H. N. Huang,   \emph{On the extension of Calabi flow on toric varieties}. {  Ann. Global Anal. Geom.} 40 (2011), no. 1,
1-19.

\bibitem{[HZ]}
H. N. Huang and K.   Zheng,
\emph{Stability of the Calabi flow near an extremal metric}.
{ Ann. Sc. Norm. Super. Pisa Cl. Sci.} (5) 11 (2012), no. 1, 167-175.

\bibitem{[LZ]}H. Z. Li and K. Zheng,  \emph{K\"ahler non-collapsing, eigenvalues and the Calabi flow}.
J. Funct. Anal. 267 (2014), no. 5, 1593-1636.



\bibitem{[LWZ]}H. Z. Li, B. Wang  and K. Zheng, \emph{Regularity scales and convergence of the Calabi flow}, J. Geom. Anal., (2018) 28, 2050-2101.

\bibitem{[LS]}Z. Q. Lu and R. Seyyedali, \emph{Remarks on a result of Chen-Cheng}, arXiv:2311.00795.












\bibitem{[ShaoSimonett]}Y. Shao and G. Simonett, \emph{Continuous maximal regularity on uniformly regular Riemannian manifolds}. J. Evol. Equ. 14 (2014), no. 1, 211--248.

\bibitem{[St1]}J. Streets,
\emph{The consistency and convergence of K-energy minimizing movements}. { Trans. Amer. Math. Soc.} 368(2016), no.7, 5075-5091.


\bibitem{[St2]}J. Streets, \emph{Long time existence of minimizing movement solutions of Calabi flow}. { Adv. Math.} 259(2014), 688-729.



\bibitem{[Stru]}M. Struwe,  \emph{Curvature flows on surfaces}. { Ann. Sc. Norm. Super. Pisa Cl. Sci.} (5) 1 (2002), no. 2,
247-274.


\bibitem{[Sz]}
G. Sz{\'e}kelyhidi,
\emph{Remark on the Calabi flow with bounded curvature}.
{  Univ. Iagel. Acta Math.} 50 (2013), 107-115.


\bibitem{[TW]}V. Tosatti and B. Weinkove, \emph{The Calabi flow with small initial energy}. Math. Res. Lett., 14(6):1033-
1039, 2007.



\bibitem{[Wang1]} B. Wang,  \emph{On the conditions to extend Ricci flow}. { Int. Math. Res. Not.} IMRN 2008, no. 8, Art. ID rnn012, 30 pp.
\bibitem{[Wang2]} B.  Wang,   \emph{On the conditions to extend Ricci flow (II)}. { Int. Math. Res. Not.} IMRN 2012, no. 14, 3192-3223.



\bibitem{[Zheng]}
 K. Zheng, \emph{Existence of constant scalar curvature K\"ahler cone metrics, properness and geodesic stability}, arXiv:1803.09506.

\bibitem{[Zheng2]}K. Zheng, \emph{Singular scalar curvature equations}, arXiv:2205.14726.

\bibitem{[Zhu]}
X. H. Zhu,  \emph{K\"ahler-Ricci soliton typed equations on compact complex manifolds with $c_1(M)>0$}. J.
Geom. Anal. 10 (2000), no. 4, 759-774.


\end{thebibliography}
\end{document}